\documentclass[onefignum,onetabnum]{siamart220329}

\usepackage{lipsum}
\usepackage{amsfonts}
\usepackage{amssymb}
\usepackage{graphicx}
\usepackage{epstopdf}
\usepackage{algpseudocode}
\usepackage{tensor}
\usepackage{verbatim}
\usepackage{mathtools}
\usepackage{cleveref}
\usepackage{framed}
\usepackage{tikz}
\usetikzlibrary{calc, shapes, angles, quotes, patterns, decorations.pathreplacing, cd, patterns.meta}
\usepackage{subcaption}
\usepackage{bm}
\usepackage{esint}
\usepackage{dsfont}
\usepackage{mathrsfs}
\usepackage{centernot}

\ifpdf
\DeclareGraphicsExtensions{.eps,.pdf,.png,.jpg}
\else
\DeclareGraphicsExtensions{.eps}
\fi

\newlength{\leftstackrelawd}
\newlength{\leftstackrelbwd}
\def\leftstackrel#1#2{\settowidth{\leftstackrelawd}%
	{${{}^{#1}}$}\settowidth{\leftstackrelbwd}{$#2$}%
	\addtolength{\leftstackrelawd}{-\leftstackrelbwd}%
	\leavevmode\ifthenelse{\lengthtest{\leftstackrelawd>0pt}}%
	{\kerneln-.5\leftstackrelawd}{}\mathrel{\mathop{#2}\limits^{#1}}}

\newcommand{\bdd}[1]{ \boldsymbol{#1} }
\newcommand{\mathbbb}[1]{\pmb{\mathbb{#1}}}
\newcommand{\unitvec}[1]{\bdd{#1}}
\makeatletter
\def\@tvsp{\mathchoice{{}\mkern-4.5mu}{{}\mkern-4.5mu}{{}\mkern-2.5mu}{}}
\def\ltrivert{\left|\@tvsp\left|\@tvsp\left|}
\def\rtrivert{\right|\@tvsp\right|\@tvsp\right|}
\makeatother

\newsiamremark{remark}{Remark}
\newsiamremark{hypothesis}{Hypothesis}
\crefname{hypothesis}{Hypothesis}{Hypotheses}
\newsiamthm{claim}{Claim}

\headers{Computing $H^2$ FEM without implementing $C^1$ elements}{M. Ainsworth and C. Parker}

\title{Computing $H^2$-conforming finite element approximations without having to implement $C^1$-elements \thanks{Submitted to the editors DATE. \funding{The second author acknowledges that this material is based upon work supported by the National Science Foundation under Award No. DMS-2201487. }}}

\author{Mark Ainsworth\thanks{ Division of Applied Mathematics, Brown University, Providence, RI
		(\email{mark\_ainsworth@brown.edu})}. \and Charles Parker \thanks{ Mathematical Institute, University of Oxford, Andrew Wiles Building, Woodstock Road, Oxford OX2 6GG, UK  (\email{charles.parker@maths.ox.ac.uk})}}

\usepackage{amsopn}

\DeclareMathOperator{\grad}{\mathbf{grad}}
\DeclareMathOperator{\dive}{div}

\DeclareMathOperator{\rot}{rot}

\makeatletter
\let\orgdescriptionlabel\descriptionlabel
\renewcommand*{\descriptionlabel}[1]{%
	\let\orglabel\label
	\let\label\@gobble
	\phantomsection
	\edef\@currentlabel{#1}%
	\let\label\orglabel
	\orgdescriptionlabel{#1}%
}
\makeatother

\ifpdf
\hypersetup{
  pdftitle={Computing H2-conforming finite element approximations without having to implement C1-elements},
  pdfauthor={M. Ainsworth and C. Parker}
}
\fi

\begin{document}

\maketitle

\begin{abstract}		
	We develop a method to compute the $H^2$-conforming finite element approximation to planar fourth order elliptic problems without having to implement $C^1$ elements. The algorithm consists of replacing the original $H^2$-conforming scheme with pre-processing and post-processing steps that require only an $H^1$-conforming Poisson type solve and an inner Stokes-like problem that again only requires at most $H^1$-conformity. We then demonstrate the method applied to the Morgan-Scott elements with three numerical examples.
\end{abstract}



\section{Introduction}
\label{sec:intro}

Conforming Galerkin finite element schemes inherit the stability properties of the underlying continuous problem. As a consequence, they provide stable approximations to a range of problem, such as structural mechanics applications, and deliver optimal approximations measured in an energy norm. Classical finite element texts \cite{Cia02} abound with examples of elements that provide $H^1$-conforming schemes (for second order problems such as the Lam\'e-Navier equations of linear elasticity) and, to a lesser extent, $H^2$-conforming schemes (for fourth order problems such as the Kirchhoff plate). 

Nevertheless, while $H^1$-conforming schemes are routinely implemented in finite element packages, $H^2$-conforming schemes are a comparative rarity. For example, Firedrake \cite{FiredrakeUserManual} and scikit-fem \cite{skfem2020} only offer the lowest order (degree five) Argyris element and FreeFEM \cite{FreeFEM} provides only the Hsieh-Clough-Tocher (HCT) element, while other major packages offer no capability for $H^2$-conforming approximation at all. Part of the reason for the $H^1$-conforming case discrepancy stems from the fact that $H^2$-conforming elements require derivative degrees of freedom that have more complicated transformation properties compared with pointwise values needed for the $H^1$-conforming case. This also means that $H^2$-conforming elements require higher order basis functions than their $H^1$-conforming counterparts. 

On the theoretical side too, there seems to be a desire to avoid $C^1$-conforming elements through the use of mixed, discontinuous Galerkin, or non-conforming finite element schemes, etc. that can be implemented using more readily available elements. Although such schemes can often be shown to be effective, they come at a price such as sacrificing the stability of the original problem in favor of an indefinite mixed finite element scheme or require the selection of appropriate penalty parameters. Either way, one ends up with a non-conforming approximation of the original $H^2$-conforming problem. While such artifacts need not present a problem \emph{per se} in the context of approximating a given equation, issues can arise when the schemes form part of larger multi-physics applications where coupling arises through terms that involve post-processing of what should be $H^2$-conforming variables or where the stability of the continuous problem is lost through the use of a mixed or non-conforming scheme. 

The principled approach to these issues is to insist on using $H^2$-conforming schemes that preserve the conformity, stability, and structure of the original problem. However, such a position fails to recognize that there are non-trivial difficulties in implementing such schemes that has fueled the large body of research on avoiding $H^2$-conformity, quite apart from the fact that few codes cater for higher order smoothness. 

The current work seeks the best of both worlds. Namely, \emph{we show how to compute the actual $H^2$-conforming approximation itself without having to implement $C^1$-conforming elements} using only approximation schemes that are routinely available in existing finite element packages. The attractions of such an approach are clear: one preserves the stability and conformity of the original formulation while enjoying the possibility to utilize existing software packages. Of course, there is no free lunch: our scheme replaces the original $H^2$-conforming scheme by a sequence of pre-processing and
post-processing steps that require only an $H^1$-conforming Poisson type solve, combined with an inner Stokes-like problem that again only requires at most $H^1$-conformity. 

After describing the problem setting in \cref{sec:problem-setting}, we explain the method first in a simplified setting in \cref{sec:simple-gammacs-connected} before turning to the general case in \cref{sec:gen-boundary}. Implementation aspects of the Stokes-like solve are addressed in \cref{sec:implementation-numerics} where we present several numerical examples. Concluding remarks appear in \cref{sec:conclusion}.

\section{Problem setting}
\label{sec:problem-setting}

Let $\Omega \subset \mathbb{R}^2$ be a simply-connected polygonal domain whose boundary  $\Gamma := \partial \Omega$, labeled as in \cref{fig:domain-example}, is partitioned into disjoint sets $\Gamma_c$, $\Gamma_s$, and $\Gamma_f$ consisting of open edges of $\Gamma$ with $|\Gamma_c \cup \Gamma_s| > 0$. We consider variational problems that take the form: 
\begin{align}
	\label{eq:biharmonic-primal-variational}
	w \in H^2_{\Gamma}(\Omega) : \qquad a(\grad w, \grad v) = F(v) \qquad \forall v \in H^2_{\Gamma}(\Omega),
\end{align}
where $F(\cdot)$ is a bounded linear functional on the space
\begin{align}
	\label{eq:htwogamma-def}
	H^2_{\Gamma}(\Omega) := \{ v \in H^2(\Omega) : v|_{\Gamma_c \cup \Gamma_s} = 0 \ \ \text{and} \ \ \partial_n v|_{\Gamma_c} = 0 \}.
\end{align}
We assume that the bilinear form $a(\cdot, \cdot)$ is bounded and positive on $\bdd{H}^1(\Omega)$ and coercive on $\grad H^2_{\Gamma}(\Omega)$; i.e. there exist positive constants $M > 0$ and $\alpha > 0$ satisfying
\begin{subequations}
	\label{eq:a-bounded-elliptic}
	\begin{alignat}{2}
		\label{eq:a-bounded}
		|a(\bdd{\theta}, \bdd{\psi})| &\leq M \|\bdd{\theta} \|_1 \| \bdd{\psi} \|_1 \qquad & &\forall \bdd{\theta}, \bdd{\psi} \in \bdd{H}^1(\Omega), \\
		\label{eq:a-positive}
		a(\bdd{\theta}, \bdd{\theta}) & \geq 0 \qquad & &\forall \bdd{\theta} \in \bdd{H}^1(\Omega), \\
		\label{eq:a-elliptic}
		a(\grad u, \grad u) &\geq \alpha \|u\|_2^2 \qquad & & \forall u \in H^2_{\Gamma}(\Omega), 
	\end{alignat}
\end{subequations}
where $\|\cdot\|$ and $\|\cdot\|_s$ denote the $L^2(\Omega)$ and $H^s(\Omega)$ norms for $s \in \mathbb{N}$. Under these conditions, the Lax-Milgram lemma can be used to show problem \cref{eq:biharmonic-primal-variational} is well-posed.

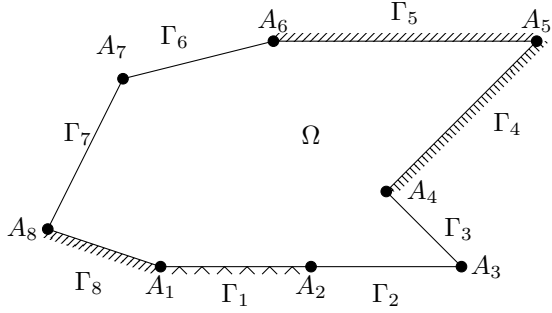
\begin{figure}[htb]
	\centering
	\begin{tikzpicture}		
		\coordinate (A1) at (-1, 0);
		\coordinate (A2) at (1, 0);
		\coordinate (A3) at (3, 0);
		\coordinate (A4) at (2, 1);
		\coordinate (A5) at (4, 3);
		\coordinate (A6) at (0.5, 3);
		\coordinate (A7) at (-1.5, 2.5);
		\coordinate (A8) at (-2.5, 0.5);
		\coordinate (Olabel) at (1, 1.75);
		
		\tikzstyle{ground}=[fill,pattern=north east 	lines,draw=none,minimum width=0.75cm,minimum height=0.3cm,inner sep=0pt,outer sep=0pt]
		
		\draw ($(A1) + (0.15, -0.1)$) -- ($(A1) + (0.25, 0)$) -- ($(A1) + (0.35, -0.1)$);
		\draw ($(A1) + (0.45, -0.1)$) -- ($(A1) + (0.55, 0)$) -- ($(A1) + (0.65, -0.1)$);
		\draw ($(A1) + (0.75, -0.1)$) -- ($(A1) + (0.85, 0)$) -- ($(A1) + (0.95, -0.1)$);
		\draw ($(A1) + (1.05, -0.1)$) -- ($(A1) + (1.15, 0)$) -- ($(A1) + (1.25, -0.1)$);
		\draw ($(A1) + (1.35, -0.1)$) -- ($(A1) + (1.45, 0)$) -- ($(A1) + (1.55, -0.1)$);
		\draw ($(A1) + (1.65, -0.1)$) -- ($(A1) + (1.75, 0)$) -- ($(A1) + (1.85, -0.1)$);
		
		\draw[ground] (A8) -- (A1) -- ($(A1) + (-0.1,-0.1)$) -- ($(A8) + (-0.1,-0.1)$) -- (A8);
		
		\draw[pattern=north west lines,draw=none,minimum width=0.75cm,minimum height=0.3cm,inner sep=0pt,outer sep=0pt] (A4) -- (A5) -- ($(A5) + (0.17, 0)$) -- ($(A4) + (0.17,0)$) -- (A4);
		
		\draw[pattern=north east lines,draw=none,minimum width=0.75cm,minimum height=0.3cm,inner sep=0pt,outer sep=0pt] (A5) -- (A6) -- ($(A6) + (0, 0.1)$) -- ($(A5) + (0,0.1)$) -- (A5);
		
		\draw (A8) -- (A1) -- (A2);
		\draw (A4) -- (A5) -- (A6);
		\draw (A2) -- (A3) -- (A4);
		\draw (A6) -- (A7) -- (A8) -- (A8);
		
		\filldraw (A1) circle (2pt) node[align=center,below]{${A}_1$};
		\filldraw (A2) circle (2pt) node[align=center,below]{${A}_2$};	
		\filldraw  (A3) circle (2pt) node[align=center,right]{${A}_3$};
		\filldraw (A4) circle (2pt) node[align=center,right]{};
		\filldraw (A5) circle (2pt) node[align=center,above]{${A}_5$};
		\filldraw (A6) circle (2pt) node[align=center,above]{${A}_6$};
		\filldraw (A7) circle (2pt) node[align=center,above]{};
		\filldraw (A8) circle (2pt) node[align=center,left]{${A}_8$};
		
		\draw ($(A4) + (0.15,0)$)  node[align=center,right]{${A}_4$};
		\draw ($(A7) + (-0.15,0.15)$)  node[align=center,above]{${A}_7$};

		\draw ($($(A1)!0.5!(A2)$) + (0, -0.1) $) node[align=center,below]{$\Gamma_1$};
		
		\draw ($($(A2)!0.5!(A3)$) + (0, -0.1) $) node[align=center,below]{$\Gamma_2$};
		
		\draw ($($(A3)!0.5!(A4)$) + (0.15, 0) $) 
		node[align=center,right]{$\Gamma_3$};
		
		\draw ($($(A4)!0.5!(A5)$) + (0.3, -0.1) $) node[align=center,right]{$\Gamma_4$};
		
		\draw ($($(A5)!0.5!(A6)$) + (0.0, 0.1) $) node[align=center,above]{$\Gamma_5$};
		
		\draw ($($(A6)!0.5!(A7)$) + (0.0, 0.3) $) node[align=center,left]{$\Gamma_6$};
		
		\draw ($($(A7)!0.5!(A8)$) + (-0.1, 0) $) node[align=center,above]{$\Gamma_7$};

		\draw ($($(A8)!0.5!(A1)$) + (-0.2, -0.2) $) node[align=center,below]{$\Gamma_8$};

		\draw (Olabel) node[align=center]{$\Omega$};
	\end{tikzpicture}
	\caption{Example domain $\Omega$ and boundary partition $\Gamma_c = \Gamma_4 \cup \Gamma_5 \cup \Gamma_8$, $\Gamma_s = \Gamma_1$, and $\Gamma_f = \Gamma_2 \cup \Gamma_3 \cup \Gamma_6 \cup \Gamma_7$.}
	\label{fig:domain-example}
\end{figure}

Despite the apparently rather special form of \cref{eq:biharmonic-primal-variational}, a number of problems that can be expressed in the form \cref{eq:biharmonic-primal-variational} do appear in practice. Perhaps the most familiar example is the biharmonic equation $\Delta^2 w = F$ whose weak form corresponds to the choice
\begin{align}
	\label{eq:a-biharmonic}
	a(\bdd{\theta}, \bdd{\psi}) = (\dive \bdd{\theta}, \dive \bdd{\psi}) \qquad \forall \bdd{\theta}, \bdd{\psi} \in \bdd{H}^1(\Omega).
\end{align}
The boundedness of $a(\cdot,\cdot)$ \cref{eq:a-bounded} with $M = \sqrt{2}$ follows from the Cauchy-Schwarz inequality, while the coercivity of $a(\grad \cdot, \grad \cdot)$ \cref{eq:a-elliptic} in the case $\Gamma = \Gamma_c$ or $\Gamma = \Gamma_s$ follows from standard arguments (see e.g. \cite[Theorem 2.2.3]{Grisvard92}).

Another application is the small deflection of an isotropic Kirchhoff plate under a transverse load $F$ (see e.g. Chapter 4 of \cite{Timoshenko59}), where  $\Omega$ represents the midsurface of a thin plate undergoing linear elastic deformation and $\Gamma_c$, $\Gamma_s$,and $\Gamma_f$ correspond to where clamped, simply supported, and free boundary conditions are applied. The rotation $\bdd{\theta}$ of the normal to the midsurface of the plate and the (symmetric) bending moments tensor $\bdd{M}$ are expressed in terms of the transverse displacement $w$ of the midsurface of the plate by the relations
\begin{align}
	\bdd{\theta}(w) = \grad w \quad \text{and} \quad \bdd{M}(\bdd{\theta}) := -D \{ (1-\nu) \bdd{\varepsilon}(\bdd{\theta}) +  \nu (\dive \bdd{\theta}) \bdd{I} \},
\end{align}
where $\bdd{\varepsilon}(\bdd{\theta}) = (\partial_j \theta_i + \partial_i \theta_j)/2$. The bending stiffness $D = E \tau^3/(12(1-\nu^2)) > 0$ of the plate depends on the Young's modulus $E > 0$ and Poisson ratio $0 < \nu \leq 1/2$ of the material, as well as the thickness $\tau > 0$ of the plate. The in-plane stresses of the plate at a distance $-\tau/2 \leq z \leq \tau/2$ above or below the midsurface of the plate are proportional to $\bdd{M}$:
\begin{align}
	\label{eq:kirchhoff-stresses}
	\bdd{\sigma} = \frac{12 z}{\tau^3} \bdd{M}(\bdd{\theta}) = \frac{12 z}{\tau^3} \bdd{M}(\grad w).
\end{align}

The transverse displacement of the plate is governed by a problem of the form \cref{eq:biharmonic-primal-variational} where
\begin{align}
	\label{eq:abilinear-def}
	a(\bdd{\theta}, \bdd{\psi}) = D\left[ (1-\nu)(\bdd{\varepsilon}(\bdd{\theta}), \bdd{\varepsilon}(\bdd{\psi})) + \nu (\dive \bdd{\theta}, \dive \bdd{\psi}) \right] \qquad \forall \bdd{\theta}, \bdd{\psi} \in \bdd{H}^1(\Omega) .
\end{align}
The boundedness of $a(\cdot, \cdot)$ \cref{eq:a-bounded} again follows from the Cauchy-Schwarz inequality, where $M$ depends on the material parameters. In the case $|\Gamma_c| > 0$,  the coercivity of $a(\grad \cdot, \grad \cdot)$ \cref{eq:a-elliptic} follows from Korn's inequality and Poincar\'{e}'s inequality:
\begin{align}
	\label{eq:a-elliptic-kirchhoff}
	a(\grad w, \grad w) \geq C \| \bdd{\varepsilon}(\grad w)\|^2 \geq C \| \grad w \|_{1}^2 \geq C \| w \|_{2}^2 \qquad \forall w \in H^2_{\Gamma}(\Omega),
\end{align}
where $C > 0$ is a positive constant. Condition \cref{eq:a-elliptic} also holds whenever $H^2_{\Gamma}(\Omega) \cap \mathcal{P}_1(\Omega) = \emptyset$ \cite[Theorem 5.9.5]{Brenner08}, and in particular when the plate is simply-supported ($\Gamma = \Gamma_s$).

Let $\mathcal{T}$ be a shape regular \cite[Definition (4.4.13)]{Brenner08} partitioning of $\Omega$ into simplices such that the nonempty intersection of any two distinct elements from $\mathcal{T}$ is a single common sub-simplex of both elements, and denote its mesh size by $h := \max_{K \in \mathcal{T}} h_K$ where $h_K := \mathrm{diam}(K)$. Let $\mathbb{W} \subset H^2(\Omega)$ denote any $H^2$-conforming finite element space on $\mathcal{T}$ and let $\mathbb{W}_{\Gamma} := \mathbb{W} \cap H^2_{\Gamma}(\Omega)$. A standard $H^2$-conforming Galerkin finite element scheme for \cref{eq:biharmonic-primal-variational} reads
\begin{align}
	\label{eq:biharmonic-primal-fem}
	w_{X} \in \mathbb{W}_{\Gamma} : \qquad a(\grad w_{X}, \grad v) = F(v) \qquad \forall v \in \mathbb{W}_{\Gamma}.
\end{align}
The use of a Galerkin scheme based on a $H^2$-conforming finite element space $\mathbb{W}_{\Gamma}$ means problem \cref{eq:biharmonic-primal-fem} inherits the stability of the original continuous problem, while a standard application of C\'{e}a's lemma shows that $w_X$ is, up to a multiplicative constant, the best approximation to $w$ in $\mathbb{W}_{\Gamma}$:
\begin{align}
	\| w - w_{X} \|_{2} \leq \frac{M}{\alpha} \inf_{v \in \mathbb{W}_{\Gamma}} \|w - v\|_{2}.
\end{align}
However, the spaces $\mathbb{W}$ are seldom employed in practice for various reasons alluded to in \cref{sec:intro}, not least of which is that few existing software packages offer $C^1$-continuous spaces, even for lower polynomial orders $p$. 

Our objective here is to seek an implementation scheme to compute $w_X$ which avoids the need to construct a $C^1$-conforming basis for the space $\mathbb{W}_{\Gamma}$ and which leads itself to implementation using standard packages.

\section{The case $\bar{\Gamma}_{c} \cup \bar{\Gamma}_s$ is connected}
\label{sec:simple-gammacs-connected}

In order to illustrate the basic idea, we first present a simple algorithm to compute the solution of \cref{eq:biharmonic-primal-fem} in the simplified setting when the portion of the boundary $\Gamma_{cs} := \bar{\Gamma}_{c} \cup \bar{\Gamma}_s$ on which $u|_{\Gamma_{cs}} = 0$ is connected. For definiteness, the space $\mathbb{W} = W^p$ is taken to be the Morgan-Scott space \cite{MorganScott75} of degree $p$:
\begin{align}
	\label{eq:wp-morgan-scott}
	W^p &:= \{ v \in C^1(\Omega) : v|_{K} \in \mathcal{P}_{p}(K) \ \forall K \in \mathcal{T} \} \subset H^2(\Omega),
\end{align}
where the degree $p$ is usually taken to be at least fifth order (although lower orders can also be considered). Later, we consider other choices for $\mathbb{W} \subset H^2(\Omega)$, but we start with the Morgan-Scott space which is generally regarded as particularly challenging to implement since, in addition to the $H^2$-conformity condition, difficulties can also arise from the topology of the mesh \cite{MorganScott75}. 

We present a finite element scheme that still produces the unique $C^1$ conforming approximation defined by \cref{eq:biharmonic-primal-fem} that one could obtain by implementing the space $W^p$, yet avoids both the need to generate a basis for the $C^1$ space and complications arising from mesh topology. The key idea is to view the Morgan-Scott space as an $H^2$-conforming subspace of a corresponding $H^1$-conforming space $\tilde{W}_{\Gamma}^{p}$ defined by
\begin{align}
	\label{eq:tildewp-morgan-scott}
	\tilde{W}_{\Gamma}^p := \{ v \in C(\Omega) : v|_{K} \in \mathcal{P}_{p}(K) \ \forall K \in \mathcal{T} \text{ and } v|_{\Gamma_{cs}} = 0 \}.
\end{align}
The fact that $\tilde{W}_{\Gamma}^p$ is only $H^1(\Omega)$-conforming means that standard packages can be used to compute the solution $\tilde{z}_p \in \tilde{W}_{\Gamma}^p$ of the problem:
\begin{align}
	\label{eq:zp-morgan-scott}
	(\grad \tilde{z}_p, \grad \tilde{v}) = F(\tilde{v}) \qquad \forall \tilde{v} \in \tilde{W}_{\Gamma}^p,
\end{align}
where for simplicity henceforth, we assume that $F$ is well-defined on  $\tilde{W}_{\Gamma}^p$. Thanks to the inclusion $\tilde{W}_{\Gamma}^p \subset W_{\Gamma}^p := W^p \cap H^2_{\Gamma}(\Omega)$, \cref{eq:biharmonic-primal-fem} may then be rewritten in the form:
\begin{align}
	\label{eq:morgan-scott-inter-with-zp}
	a(\grad w_p, \grad v) = (\grad \tilde{z}_p, \grad v) \qquad \forall v \in W_{\Gamma}^p,
\end{align}
where we use the notation $w_p := w_X$ to highlight the dependence on $p$. Observing that \cref{eq:morgan-scott-inter-with-zp} only involves gradients of functions in $W_{\Gamma}^p$, it is tempting to instead pose \cref{eq:morgan-scott-inter-with-zp} over the image space $\grad W_{\Gamma}^p$ and seek $\bdd{\theta}_{p-1} = \grad w_p$ which, in view of \cref{eq:morgan-scott-inter-with-zp}, must satisfy
\begin{align}
	\label{eq:morgan-scott-inter-a-conds}
	\bdd{\theta}_{p-1} \in \grad W_{\Gamma}^p : \qquad 
	a(\bdd{\theta}_{p-1}, \bdd{\psi}) = (\grad \tilde{z}_p, \bdd{\psi})  \qquad \forall \bdd{\psi} \in \grad W_{\Gamma}^p.
\end{align}
An immediate benefit of this approach is that $\grad W_{\Gamma}^p$ is a conforming subspace of 
\begin{align}
	\label{eq:thetagamma-def}
	\bdd{\Theta}_{\Gamma}(\Omega) &:= \{ \bdd{\theta} \in \bdd{H}^1(\Omega) : \bdd{\theta} |_{\Gamma_c} = \bdd{0} \text{ and } \unitvec{t} \cdot \bdd{\theta}|_{\Gamma_s} = 0   \},
\end{align}
for which only $\bdd{H}^1$-conformity is required. The snag in working with the space $\grad W_{\Gamma}^p$ is that, while the continuity requirements are weakened, the structure of the space means that one has to find a basis consisting of curl-free functions. Such a basis can be constructed by taking the gradients of a basis for the Morgan-Scott space which, of course, defeats the object. Instead, we relax the curl-free constraint and seek $\bdd{\theta}_{p-1}$ in the larger subspace
\begin{align}
	\label{eq:thetap-morgan-scott}
	\bdd{G}^{p-1}_{\Gamma} := \{ \bdd{\theta} \in \bdd{\Theta}_{\Gamma}(\Omega) : \bdd{\theta}|_{K} \in [\mathcal{P}_{p-1}(K)]^2 \ \forall K \in \mathcal{T} \}
\end{align}
for which it is straightforward to construct a basis: only $\bdd{C}^0$-continuity is required. However, relaxing the curl-free constraint comes at the expense of introducing an indeterminancy in $\bdd{\theta}_{p-1}$ stemming from the fact that $\grad W_{\Gamma}^p \neq \bdd{G}^{p-1}_{\Gamma}$ along with the fact that \cref{eq:morgan-scott-inter-a-conds} does not determine the components of $\bdd{\theta}_{p-1} \in \bdd{G}^{p-1}_{\Gamma} / \grad W_{\Gamma}^p$. Still undeterred, we seek additional conditions on $\bdd{\theta}_{p-1}$ that remove this non-uniqueness.

To this end, we observe that the spaces $W^p_{\Gamma}$ and $\bdd{G}^{p-1}_{\Gamma}$ form part of an \textit{exact sequence} (see \cref{lem:vp-specify-constants}):
\begin{align}
	\label{eq:stokes-complex-bcs-fem}
	0 \xrightarrow{\ \ \ \subset \ \ \ } W^p_{\Gamma} \xrightarrow{\ 	\ \ \grad \ \ \ } \bdd{G}^{p-1}_{\Gamma} \xrightarrow{ \ \ \rot \ \ } \rot \bdd{G}^{p-1}_{\Gamma} \xrightarrow{ \ \ \ 0 \ \ \ } 0,
\end{align}
where $\rot \bdd{\theta} = \partial_x \theta_2 - \partial_y \theta_1$. While $\grad W_{\Gamma}^p$ is a proper subspace of $\bdd{G}^{p-1}_{\Gamma}$, the sequence \cref{eq:stokes-complex-bcs-fem} is \textit{exact} in the sense the kernel of each operator is the range of the preceding operator; i.e. if $\bdd{\theta} \in \bdd{G}^{p-1}_{\Gamma}$ satisfies $\rot \bdd{\theta} \equiv 0$, then $\bdd{\theta} = \grad w$ for some $w \in W^p_{\Gamma}$. In other words, the following characterization holds:
\begin{align}
	\label{eq:morgan-scott-rotfree-grad}
	\bdd{\theta}_{p-1} \in \grad W_{\Gamma}^p \iff \rot \bdd{\theta}_{p-1} \equiv 0.
\end{align}
The significance of \cref{eq:morgan-scott-rotfree-grad} is that the condition $\bdd{\theta}_{p-1} \in \grad W_{\Gamma}^p$ can be enforced by augmenting \cref{eq:morgan-scott-inter-a-conds} with the additional, independent, condition that $\rot \bdd{\theta}_{p-1} \equiv 0$. This is most conveniently achieved through the use of a Lagrange multiplier $r_{p-2} \in \rot \bdd{\Theta}_{\Gamma}^{p-1}$, which gives rise to the following variational problem: Find $(\bdd{\theta}_{p-1}, r_{p-2}) \in \bdd{G}^{p-1}_{\Gamma} \times \rot \bdd{G}^{p-1}_{\Gamma}$ such that
\begin{subequations}
	\label{eq:stokes-system-morgan-scott}
	\begin{alignat}{2}
		\label{eq:stokes-system-morgan-scott-1}
		a(\bdd{\theta}_{p-1}, \bdd{\psi}) + (\rot \bdd{\psi}, r_{p-2}) &= (\grad \tilde{z}_p, \bdd{\psi}) \qquad & &\forall \bdd{\psi} \in \bdd{G}^{p-1}_{\Gamma}, \\
		\label{eq:stokes-system-morgan-scott-2}
		(\rot \bdd{\theta}_{p-1}, s) &= 0 \qquad & &\forall s \in \rot \bdd{G}^{p-1}_{\Gamma}.
	\end{alignat}
\end{subequations}
\Cref{lem:morgan-scott-simplified} below states that the system \cref{eq:stokes-system-morgan-scott} uniquely determines $\bdd{\theta}_{p-1}$, and guarantees that $\bdd{\theta}_{p-1} = \grad w_p$ while only requiring the use of $\bdd{C}^0$-continuous finite element spaces. Of course, the reader may quite reasonably object that \cref{eq:stokes-system-morgan-scott} only enables us to obtain the \textit{gradient} of the Morgan-Scott approximation $w_p$ of the displacement.

Remarkably, the \textit{$H^2$-conforming approximation to the actual displacement $w_p \in W_{\Gamma}^p \subset \tilde{W}_{\Gamma}^p$ can be recovered via a post-processing step that again only involves the $H^1$-conforming subspace} $\tilde{W}_{\Gamma}^p$ (rather than the $H^2$-conforming space $W_{\Gamma}^p$) as follows:
\begin{align}
	\label{eq:morgan-scott-final-projection}
	\tilde{w}_p \in \tilde{W}_{\Gamma}^p : \qquad (\grad \tilde{w}_p, \grad v) = (\bdd{\theta}_{p-1}, \grad v) \qquad \forall v \in \tilde{W}_{\Gamma}^p.
\end{align}
The following result confirms that the displacement $\tilde{w}_p \in \tilde{W}_{\Gamma}^p$ computed from the $H^1$-conforming scheme \cref{eq:morgan-scott-final-projection} is well-defined and \textit{coincides with the $H^2$-conforming Morgan-Scott displacement $w_p$} defined by \cref{eq:biharmonic-primal-fem}; i.e. $w_p = \tilde{w}_p$:
\begin{lemma}
	\label{lem:morgan-scott-simplified}
	Let $W_{\Gamma}^p$, $\tilde{W}_{\Gamma}^p$, and $\bdd{G}^{p-1}_{\Gamma}$ be given by \cref{eq:wp-morgan-scott,eq:tildewp-morgan-scott,eq:thetap-morgan-scott}, respectively, and let $\Gamma_{cs}$ be connected. Then, there exists unique $(\tilde{z}_p, \bdd{\theta}_{p-1}, r_{p-2}, \tilde{w}_p) \in \tilde{W}_{\Gamma}^p \times \bdd{G}^{p-1}_{\Gamma} \times \rot \bdd{G}^{p-1}_{\Gamma} \times \tilde{W}_{\Gamma}^p$ satisfying \cref{eq:zp-morgan-scott,eq:stokes-system-morgan-scott,eq:morgan-scott-final-projection}. Moreover, $\grad \tilde{w}_p = \bdd{\theta}_{p-1}$ and $\tilde{w}_p = w_p$, where $w_p$ is defined by \cref{eq:biharmonic-primal-fem}.
\end{lemma}
\Cref{lem:morgan-scott-simplified} is a consequence of \cref{thm:morgan-scott-simplified-ext} and \cref{lem:vp-specify-constants} which we discuss in the next section.

In summary, when $\Gamma_{cs}$ is connected, the $H^2$-conforming Morgan-Scott finite element approximation \cref{eq:biharmonic-primal-fem} can be computed by solving a sequence of three standard finite element problems:
\begin{enumerate}
	\item Pre-process the data $F$ by computing the $H^1$-conforming Riesz representer of the source data on $\tilde{W}_{\Gamma}^p$ \cref{eq:zp-morgan-scott}.
	
	\item Solve a Stokes-like problem using the mixed Scott-Vogelius element \cite{ScottVog84,ScottVog85} $\bdd{G}^{p-1}_{\Gamma} \times \rot \bdd{G}^{p-1}_{\Gamma}$ \cref{eq:stokes-system-morgan-scott}. 
	
	\item Post-process the gradient $\bdd{\theta}_{p-1}$ by computing an $H^1$-conforming elliptic projection on $\tilde{W}_{\Gamma}^p$ \cref{eq:morgan-scott-final-projection}.
\end{enumerate}
Crucially, steps 1-3 only require $H^1$-conforming finite element spaces that are commonly available in software packages. The limitation of the analysis to $\Gamma_{cs}$ connected here is unduly restrictive. In the next section, we modify the basic approach in this section to accommodate more general boundary conditions.

\section{General boundary conditions}
\label{sec:gen-boundary}

The foregoing discussion dealt with the case when $\Gamma_{cs} = \bar{\Gamma}_c \cup \bar{\Gamma}_s$ is simply-connected for reasons that only become in the general case in which $\Gamma_{cs} = \cup_{i=1}^{N} \Gamma_{cs}^{(i)}$ consists of $N \geq 1$ connected components. If $\mathbb{W}$ is again taken to be the Morgan-Scott space $W^p$ as in the previous section, then the same arguments of the previous section again show that the quantity $\bdd{\theta}_{p-1} \in \bdd{G}^{p-1}_{\Gamma}$  must satisfy
\begin{align}
	\label{eq:thetap-necessary-conds}
	\rot \bdd{\theta}_{p-1} \equiv 0 \quad \text{and} \quad a(\bdd{\theta}_{p-1}, \bdd{\psi}) = (\grad \tilde{z}_p, \bdd{\psi}) \qquad \forall \bdd{\psi} \in \grad W_{\Gamma}^p.
\end{align}
When $\Gamma_{cs}$ consists of a single connected component ($N=1$), the exactness of the sequence \cref{eq:stokes-complex-bcs-fem} meant that the reverse implication in \cref{eq:morgan-scott-rotfree-grad} holds so that conditions \cref{eq:thetap-necessary-conds} uniquely determine $\bdd{\theta}_{p-1}$. However, when $\Gamma_{cs}$ is not connected ($N > 1$), the corresponding implication is no longer valid:
\begin{align*}
	\rot \bdd{\theta}_{p-1} \equiv 0 \centernot\implies \bdd{\theta}_{p-1} \in \grad W_{\Gamma}^p.
\end{align*}
In other words, $\bdd{\Theta}_{\Gamma}^{p-1}$ contains rot-free fields that cannot be written as the gradient of a potential in $W_{\Gamma}^p$ and, as a consequence, \cref{lem:morgan-scott-simplified} does not extend to the case when $\Gamma_{cs}$ is not connected. We justify this claim below, but it is important to note that this issue is not specific to the Morgan-Scott space. In this section, we aim to extend the basic idea used in \cref{sec:simple-gammacs-connected} where (a) $\Gamma_{cs}$ need not be simply-connected and (b) $\mathbb{W}$ can be chosen to be a general $H^2$-conforming finite element space.

\subsection{General Theory}
\label{sec:gen-theory}

Let $\mathbb{W} \subseteq H^2(\Omega)$ be any $H^2$-conforming subspace and set $\mathbb{W}_{\Gamma} := \mathbb{W} \cap H^2_{\Gamma}(\Omega)$. Our objective is again to compute the $H^2$-conforming approximation $w_X$ defined in \cref{eq:biharmonic-primal-fem} while avoiding having to implement the space $\mathbb{W}$ directly. Let $\mathbbb{G}_{\Gamma} \subseteq \bdd{\Theta}_{\Gamma}(\Omega)$ denote any conforming subspace such that  $\grad \mathbb{W}_{\Gamma} \subseteq \mathbbb{G}_{\Gamma}$ and which satisfies the following property: \\
\begin{description}
	\item[(A1)\label{hp:wp-constants-boundary-assumption}] Given $\{ \mu_i \}_{i=1}^{N} \subset \mathbb{R}$, there exists $v \in \mathbb{W}$ satisfying
	\begin{align}
		\label{eq:wp-constants-boundary-assumption}
		v|_{\Gamma_{cs}^{(i)}} = \mu_i, \qquad 1 \leq i \leq N, \qquad  \partial_n v|_{\Gamma} = 0, \quad \text{and} \quad \grad v \in \mathbbb{G}_{\Gamma}. 	
	\end{align} 
	\vspace{0.125em}
\end{description}
With this notation, the generalized version of \cref{eq:zp-morgan-scott} reads
\begin{align}
	\label{eq:zp-abstract}
	\tilde{z}_X \in \tilde{\mathbb{W}}_{\Gamma} : \qquad (\grad \tilde{z}_X, \grad \tilde{v}) = F(\tilde{v}) \qquad \forall \tilde{v} \in \tilde{\mathbb{W}}_{\Gamma},
\end{align}
where $\tilde{\mathbb{W}}_{\Gamma}$ is chosen to be any space satisfying the condition \\
\begin{description}
	\item[(A2)\label{hp:tildewp-condition}] $\mathbb{W}_{\Gamma} \subseteq \tilde{\mathbb{W}}_{\Gamma} \subset H^1(\Omega).$ \\
\end{description}
As in the previous section, we will again assume that $F$ is well-defined on $\tilde{\mathbb{W}}_{\Gamma}$. Assumption \ref{hp:tildewp-condition} offers the flexibility to choose $\tilde{\mathbb{W}}_{\Gamma}$ to be a space that requires only $H^1$-conformity so that \cref{eq:zp-abstract} can be implemented using only $C^0$ finite element spaces. As before, $\bdd{\theta}_X := \grad w_X$ satisfies the following analogue of \cref{eq:thetap-necessary-conds}:
\begin{align}
	\label{eq:thetap-necessary-conds-abstract}
	\rot \bdd{\theta}_X \equiv 0 \quad \text{and} \quad a(\bdd{\theta}_X, \bdd{\psi}) = (\grad \tilde{z}_X, \bdd{\psi}) \qquad \forall \bdd{\psi} \in \grad \mathbb{W}_{\Gamma}.
\end{align}

\begin{figure}[htb]
	\centering
	\begin{tikzpicture}	[scale=0.5]	
		\coordinate (A1) at (-8, 7);
		\coordinate (A2) at (-11, 7);
		\coordinate (A3) at (-14, 5);
		\coordinate (A4) at (-12, 3);
		\coordinate (A5) at (-14, 1);
		\coordinate (A6) at (-12, -1.5);
		\coordinate (A7) at (-7.5, -1.5);
		\coordinate (A75) at (-5, 0);
		\coordinate (A80) at (-5, 2);
		\coordinate (A8) at (-4.5, 3);
		\coordinate (A9) at (-5, 6);	
		
		\draw[line width=1mm] (A1) -- (A2);
		\draw[line width=1mm] (A3) -- (A4) -- (A5);
		\draw[line width=1mm] (A6) -- (A7);
		\draw[line width=1mm] (A8) -- (A9);
		
		\draw[line width=0.1mm] (A2) -- (A3);
		\draw[line width=0.1mm] (A5) -- (A6);
		\draw[line width=0.1mm] (A7) -- (A75);

		\draw[line width=0.1mm] (A80) -- (A8);
		\draw[draw=none] (A75) -- (A80) node[pos=0.5, sloped]{...};
		\draw[line width=0.1mm] (A9) -- (A1);
		
		\filldraw (A1) circle (4pt);
		\filldraw (A2) circle (4pt);	
		\filldraw  (A3) circle (4pt);
		\filldraw (A4) circle (4pt);
		\filldraw (A5) circle (4pt);
		\filldraw (A6) circle (4pt);
		\filldraw (A7) circle (4pt);
		\filldraw (A8) circle (4pt);
		\filldraw (A9) circle (4pt);
		
		\draw ($($(A1)!0.5!(A2)$) + (0, 0) $) node[align=center,above]{$\Gamma_{cs}^{(1)}$};
		
		\draw ($($(A2)!0.5!(A3)$) + (0, 0) $) node[align=center,above]{$\Gamma_{f}^{(1)}$};
		
		\draw ($(A4) +(-1, 0)$) node[align=center,left]{$\Gamma_{cs}^{(2)}$};
		
		\draw ($($(A5)!0.5!(A6)$) + (-0.1, -0.5) $) node[align=center,left]{$\Gamma_{f}^{(2)}$};
		
		\draw ($($(A6)!0.5!(A7)$) + (0, 0) $) node[align=center,below]{$\Gamma_{cs}^{(3)}$};
		
		\draw ($($(A8)!0.5!(A9)$) + (0, 0) $) node[align=center,right]{$\Gamma_{cs}^{(N)}$};
		
		\draw ($($(A9)!0.5!(A1)$) + (0.3, 0) $) node[align=center,above]{$\Gamma_{f}^{(N)}$};
	\end{tikzpicture}
	\caption{Notation and ordering of $\Gamma_{cs}^{(i)}$ and $\Gamma_f^{(i)}$, $i = 1,\ldots,N$.}
	\label{fig:domain-example-bcs}
\end{figure}
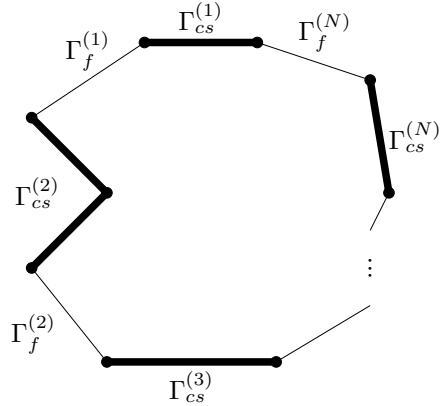

Assumption \ref{hp:wp-constants-boundary-assumption} means that the $\rot$ operator is invertible on the space $\mathbbb{G}_{\Gamma}$ in the following sense:
\begin{lemma}
	\label{lem:invert-rot-with-free-averages}
	Let $\mathbb{W}$ and $\mathbbb{G}_{\Gamma}$ satisfy \ref{hp:wp-constants-boundary-assumption}. Then, for every $r \in \rot \mathbbb{G}_{\Gamma}$ and $\vec{\omega} \in \mathbb{R}^{N}$ such that
	\begin{align}
		\label{eq:kappaij-constraint}
		\sum_{i=1}^{N} \omega_{i} = \int_{\Omega} r \ d\bdd{x},
	\end{align}
	there exists $\bdd{\theta} \in \mathbbb{G}_{\Gamma}$ satisfying
	\begin{align}
		\label{eq:invert-rot-averages-early}
		\rot \bdd{\theta} = r \quad \text{and} \quad \int_{\Gamma_f^{(i)}} \unitvec{t} \cdot \bdd{\theta} \ ds = \omega_i, \qquad 1 \leq i \leq N,
	\end{align}
	where $\{ \Gamma_{f}^{(i)} \}_{i=1}^{N}$ denote the $N$ connected components of $\Gamma_f = \partial \Omega \setminus \Gamma_{cs}$ that separate the $N$ connected components of $\Gamma_{cs}$ (e.g. in \cref{fig:domain-example-bcs}).
\end{lemma}
\begin{proof}
	Let $r \in \rot \mathbbb{G}_{\Gamma}$ and $\vec{\omega} \in \mathbb{R}^{N}$ satisfying \cref{eq:kappaij-constraint} be given. By labeling the components $\Gamma_f^{(i)}$ and $\Gamma_{cs}^{(i)}$ appropriately, we may arrange $\Gamma_f^{(i)}$ to be located between $\Gamma_{cs}^{(i)}$ and $\Gamma_{cs}^{(i+1)}$, $i = 1,2,\ldots, N-1$, and for the components to be ordered counterclockwise as in \cref{fig:domain-example-bcs} so that $\Gamma_{cs}^{(1)}$ and $\Gamma_{f}^{(N)}$ share a common endpoint. Let $\bdd{\psi} \in \mathbbb{G}_{\Gamma}$ be any function satisfying $\rot \bdd{\psi} = r$. Thanks to \ref{hp:wp-constants-boundary-assumption}, there exists $v \in \mathbb{W}$ satisfying $\partial_n v|_{\Gamma} = 0$ and $v|_{\Gamma_{cs}^{(i)}} = \mu_i$, $i=1,2,\ldots, N$, with $\grad v \in \mathbbb{G}_{\Gamma}$, where $\mu_1 = 0$ and
	\begin{align*}
		\mu_{i+1} = \mu_i + \omega_i - \int_{\Gamma_f^{(i)}} \unitvec{t} \cdot \bdd{\psi} \ ds, \qquad 1 \leq i \leq N.
	\end{align*}
	These relations imply that
	\begin{align}
		\label{eq:proof:mu_np1_vanish}
		\mu_{N+1} = \sum_{i=1}^{N} (\mu_{i+1} - \mu_{i}) 
		= \sum_{i=1}^{N} \omega_i - \int_{\Gamma} \unitvec{t} \cdot \bdd{\psi} \ ds = \sum_{i=1}^{N} \omega_i - \int_{\Omega} \rot \bdd{\psi} \ d\bdd{x} = 0 = \mu_1,
	\end{align}
	where we used Stokes Theorem, that $\rot \bdd{\psi} = r$, and the compatibility condition \cref{eq:kappaij-constraint}.  
	Thus, $\bdd{\theta} := \bdd{\psi} + \grad v$ satisfies $\bdd{\theta} \in \mathbbb{G}_{\Gamma}$ and $\rot \bdd{\theta} = \rot \bdd{\psi} = r$. Moreover, there holds
	\begin{align*}
		\int_{\Gamma_f^{(i)}} \unitvec{t} \cdot \bdd{\theta} \ ds = \int_{\Gamma_f^{(i)}} \unitvec{t} \cdot \bdd{\psi} \ ds + v|_{\Gamma_{cs}^{(i+1)}} - v|_{\Gamma_{cs}^{(i)}} = \int_{\Gamma_f^{(i)}} \unitvec{t} \cdot \bdd{\psi} \ ds + \mu_{i+1} - \mu_i = \omega_i,
	\end{align*}
	for $1 \leq i \leq N$, where we use the convention $\Gamma_{cs}^{(N+1)} := \Gamma_{cs}^{(1)}$ and we used \cref{eq:proof:mu_np1_vanish} and the identity
	\begin{align}
		\label{eq:proof:ftc-wp}
		\int_{\Gamma_f^{(i)}} \unitvec{t} \cdot \grad v \ ds = v|_{\Gamma_{cs}^{(i+1)}} - v|_{\Gamma_{cs}^{(i)}}, \qquad 1 \leq i \leq N.
	\end{align}
\end{proof}

\Cref{lem:invert-rot-with-free-averages} shows that there exist rot-free fields $\bdd{\theta} \in \mathbbb{G}_{\Gamma}$ that cannot be expressed as the gradient of a potential $v \in \mathbb{W}_{\Gamma}$. A counterexample can be constructed as follows. Choose $r \equiv 0$ and, when $\Gamma_{cs}$ is not connected ($N > 1$), let $\omega_i$ (not all zero) satisfy $\sum_{i=1}^{N} \omega_i = 0$. \Cref{lem:invert-rot-with-free-averages} asserts the existence of a rot-free function $\bdd{\theta} \in \mathbbb{G}_{\Gamma}$ satisfying $\int_{\Gamma_{f}^{(i)}} \unitvec{t} \cdot \bdd{\theta} = \omega_i$. Moreover, every $v \in \mathbb{W}_{\Gamma}$ satisfies $v|_{\Gamma_{cs}} = 0$ so that $\int_{\Gamma_f^{(i)}} \unitvec{t} \cdot \grad v \ ds = 0$  and hence $\bdd{\theta} \notin \grad \mathbb{W}_{\Gamma}$. In other words, \cref{eq:morgan-scott-rotfree-grad} no longer holds in the case when $\Gamma_{cs}$ is not connected ($N > 1$) and the arguments used in \cref{sec:simple-gammacs-connected} break down.

Nevertheless, this counterexample suggests that the following modified version of \cref{eq:morgan-scott-rotfree-grad} may hold: 
\begin{align}
	\label{eq:morgan-scott-rot0-gradv}
	\bdd{\theta}\in \mathbbb{G}_{\Gamma} : \quad \rot \bdd{\theta} \equiv 0 \iff \bdd{\theta} \in \grad \mathbb{V}_{\Gamma},
\end{align}
where
\begin{align}
	\label{eq:vp-def-morganscott}
	\mathbb{V}_{\Gamma} := \{ v \in \mathbb{W} : v|_{\Gamma_{cs}^{(1)}} = 0, \ v|_{\Gamma_{cs}^{(i)}} \in \mathbb{R}, \ 2 \leq i \leq N, \text{ and } \partial_n v|_{\Gamma_{c} } = 0 \}.
\end{align}
Condition \cref{eq:morgan-scott-rot0-gradv} means that replacing $\mathbb{W}_{\Gamma}$ by $\mathbb{V}_{\Gamma}$ in the sequence \cref{eq:stokes-complex-bcs-fem} generates an exact sequence, which leads us to make an additional assumption: \\

\begin{description}
	\item[(A3)\label{hp:stokes-complex-bcs-discon-fem}] The following sequence is exact:
	\begin{align}
		\label{eq:stokes-complex-bcs-discon-fem}
		0 \xrightarrow{\ \ \ \subset \ \ \ } \mathbb{V}_{\Gamma} \xrightarrow{\ 	\ \ \grad \ \ \ } \mathbbb{G}_{\Gamma} \xrightarrow{ \ \ \rot \ \ } \rot \mathbbb{G}_{\Gamma} \xrightarrow{ \ \ \ 0 \ \ \ } 0. 
	\end{align}
	\vspace{0.125em}
\end{description}

The space $\mathbb{V}_{\Gamma}$ differs from $\mathbb{W}_{\Gamma}$ in that it admits functions whose traces on $\Gamma_{cs}^{(i)}$, $i > 1$, are (arbitrary) constants, whereas $\mathbb{W}_{\Gamma}$ imposes constraints on these spaces, viz:
\begin{align}
	\label{eq:wp-integral-zero-chara}
	\mathbb{W}_{\Gamma} = \left\{ w \in \mathbb{V}_{\Gamma} : \int_{\Gamma_{f}^{(i)}} \unitvec{t} \cdot \grad w \ ds = 0, \ 1 \leq i \leq N-1 \right\}.
\end{align}
Characterizing $\mathbb{W}_{\Gamma}$ as a subspace of $\mathbb{V}_{\Gamma}$ in this way is helpful in removing the indeterminancy in $\bdd{\theta}_X$ present in \cref{eq:thetap-necessary-conds-abstract}. In particular, the identity \cref{eq:proof:ftc-wp} means that if a function $w \in \mathbb{V}_{\Gamma}$ satisfies the conditions appearing in \cref{eq:wp-integral-zero-chara}, then $w|_{\Gamma_{cs}^{(i)}} = w|_{\Gamma_{cs}^{(1)}} = 0$ for all $i$, and hence $w \in \mathbb{W}_{\Gamma}$. Consequently, we can ensure that $\bdd{\theta}_X = \grad w_X$ belongs to $\grad \mathbb{W}_{\Gamma}$ by enforcing the constraints
\begin{align}
	\label{eq:thetap-gammaf-integral-constr}
	\int_{\Gamma_{f}^{(i)}} \unitvec{t} \cdot \bdd{\theta}_{X}  \ ds = 0, \qquad 1 \leq i \leq N-1.
\end{align}
These constraints are imposed using Lagrange multipliers which, in conjunction with again treating the rotation free constraint in \cref{eq:thetap-necessary-conds} as in the previous section, gives the following generalization of the scheme \cref{eq:zp-morgan-scott,eq:stokes-system-morgan-scott,eq:morgan-scott-final-projection} to the case $\Gamma_{cs}$ is not connected ($N > 1$): 

\begin{algorithm}[htb]
	\caption{}
	\label{alg:abstract-method}
	\begin{algorithmic}[1]
		\Ensure{$(\tilde{z}_X, \bdd{\theta}_{X}, r_{X}, \vec{\kappa}, \tilde{w}_X) \in \tilde{\mathbb{W}}_{\Gamma} \times \mathbbb{G}_{\Gamma} \times \rot \mathbbb{G}_{\Gamma} \times \mathbb{R}^{N-1} \times \tilde{\mathbb{W}}_{\Gamma}$}
		
		\State{Pre-process the data $F$:
			\begin{align}
				\label{eq:zp-morgan-scott-2}
				(\grad \tilde{z}_X, \grad \tilde{v}) &= F(\tilde{v}) \qquad \forall \tilde{v} \in \tilde{\mathbb{W}}_{\Gamma}.
		\end{align}
		}
		
		\State{Solve a Stokes-like problem: 
			\begin{subequations}
				\label{eq:stokes-system-gen-morgan-scott}
				\begin{alignat}{2}
					\label{eq:stokes-system-gen-morgan-scott-1}
					a(\bdd{\theta}_{X}, \bdd{\psi}) + (\rot \bdd{\psi}, r_{X}) + \sum_{i=1}^{N-1} (\unitvec{t} \cdot \bdd{\psi}, \kappa_i)_{\Gamma_f^{(i)}} &= (\grad \tilde{z}_X, \bdd{\psi}) \qquad & &\forall \bdd{\psi} \in \mathbbb{G}_{\Gamma}, \\
					\label{eq:stokes-system-gen-morgan-scott-2}
					(\rot \bdd{\theta}_{X}, s)  &= 0 \qquad & &\forall s \in \rot \mathbbb{G}_{\Gamma}, \\
					\label{eq:stokes-system-gen-morgan-scott-3}
					\sum_{i=1}^{N-1}(\unitvec{t} \cdot \bdd{\theta}_{X}, \mu_i)_{\Gamma_f^{(i)}} &= 0 \qquad & &\forall \vec{\mu} \in \mathbb{R}^{N-1}.
				\end{alignat}
		\end{subequations}
		}
	
		\State{Post-process the gradient $\bdd{\theta}_{X}$:
			\begin{align}
				\label{eq:morgan-scott-final-projection-2}
				(\grad \tilde{w}_X, \grad v) &= (\bdd{\theta}_{X}, \grad v) \qquad \forall v \in \tilde{\mathbb{W}}_{\Gamma}.
		\end{align}}
	\end{algorithmic}
\end{algorithm}

In the case where $\Gamma_{cs}$ is connected ($N=1$), the sums in \cref{eq:stokes-system-gen-morgan-scott} are not present and \cref{alg:abstract-method} reduces to \cref{eq:zp-morgan-scott,eq:stokes-system-morgan-scott,eq:morgan-scott-final-projection}. Similarly to the previous section, \cref{eq:stokes-system-gen-morgan-scott} implies that $\bdd{\theta}_{X} = \grad w_X$ satisfies \cref{eq:morgan-scott-inter-with-zp}. Equally well, the displacement $w_X \in \mathbb{W}_{\Gamma}$ is again given by $w_X = \tilde{w}_X$, where $\tilde{w}_X$ is obtained by post-processing as in \cref{eq:morgan-scott-final-projection-2}. The following generalization of \cref{lem:morgan-scott-simplified} holds:
\begin{theorem}
	\label{thm:morgan-scott-simplified-ext}
	Let $\mathbb{W} \subseteq H^2(\Omega)$, $\mathbbb{G}_{\Gamma} \subseteq \bdd{\Theta}_{\Gamma}(\Omega)$, and $\tilde{\mathbb{W}}_{\Gamma} \subset H^1(\Omega)$ be any conforming subspaces satisfying conditions \ref{hp:wp-constants-boundary-assumption}, \ref{hp:tildewp-condition}, and \ref{hp:stokes-complex-bcs-discon-fem}. Then, \cref{alg:abstract-method} delivers a unique function $\tilde{w}_X \in \mathbb{W}_{\Gamma}$ that satisfies $\grad \tilde{w}_X = \bdd{\theta}_{X}$ and $\tilde{w}_X = w_X$, where $w_X$ is the $H^2$-conforming $\mathbb{W}_{\Gamma}$ approximation defined by \cref{eq:biharmonic-primal-fem}.
\end{theorem}
\begin{proof}
	Assume first that $(\tilde{z}_X, \bdd{\theta}_{X}, r_{X}, \vec{\kappa}, \tilde{w}_X)$ appearing in \cref{alg:abstract-method} exist and are unique. Choosing $s = \rot \bdd{\theta}_{X}$ in \cref{eq:stokes-system-gen-morgan-scott-2} and $\mu_i = (\unitvec{t} \cdot \bdd{\theta}_{X}, 1)_{\Gamma_{f}^{(i)}}$ in \cref{eq:stokes-system-gen-morgan-scott-3} gives $\rot \bdd{\theta}_{X} \equiv 0$ and $(\unitvec{t} \cdot \bdd{\theta}_{X}, 1)_{\Gamma_{f}^{(i)}} = 0$, $1 \leq i \leq N-1$. Consequently, $\bdd{\theta}_{X} = \grad u_X$ for some $u_X \in \mathbb{W}_{\Gamma}$ by the exactness of \cref{eq:stokes-complex-bcs-discon-fem} and identity \cref{eq:wp-integral-zero-chara}. Thus, \cref{eq:morgan-scott-final-projection-2} reads
	\begin{align*}
		(\grad \tilde{w}_X, \grad v) = (\grad u_X, \grad v) \qquad \forall v \in \tilde{\mathbb{W}}_{\Gamma}.
	\end{align*} 
	Since $u_X \in \mathbb{W}_{\Gamma} \subseteq \tilde{\mathbb{W}}_{\Gamma}$, $\tilde{w}_X = u_X$. Choosing $\bdd{\psi} \in \grad \mathbb{W}_{\Gamma}$ in \cref{eq:stokes-system-gen-morgan-scott-2} and using \cref{eq:zp-morgan-scott-2} then gives
	\begin{align*}
		a(\grad \tilde{w}_X, \grad v) = (\grad \tilde{z}_X, \grad v) = F(v) \qquad \forall v \in \mathbb{W}_{\Gamma},
	\end{align*}
	and hence, thanks to the uniqueness of \cref{eq:biharmonic-primal-fem}, we conclude that $\tilde{w}_X = w_X$.	
	
	We now show that $(\tilde{z}_X, \bdd{\theta}_{X}, r_{X}, \vec{\kappa}, \tilde{w}_X)$ appearing in \cref{alg:abstract-method} exist and are unique. It suffices to show that the only solution to \cref{eq:zp-morgan-scott-2,eq:stokes-system-gen-morgan-scott,eq:morgan-scott-final-projection-2} in the case $F \equiv 0$, vanishes identically. If $F \equiv 0$, then \cref{eq:zp-morgan-scott-2} means that $\tilde{z}_X \equiv 0$. 
	As a result, the triplet $(\bdd{\theta}_{X}, r_{X}, \vec{\kappa}) \in \mathbbb{G}_{\Gamma} \times \rot \mathbbb{G}_{\Gamma} \times \mathbb{R}^{N-1}$ satisfies \cref{eq:stokes-system-gen-morgan-scott} with $\tilde{z}_X \equiv 0$. As shown above, $\bdd{\theta}_{X} = \grad u_X$ for some $u_X \in \mathbb{W}_{\Gamma}$.
	 Choosing $\bdd{\psi} = \grad u_X$ in \cref{eq:stokes-system-gen-morgan-scott-1} and using the ellipticity of $a(\grad \cdot,\grad \cdot)$ \cref{eq:a-elliptic}, we conclude $u_X \equiv 0$. Thanks to \cref{lem:invert-rot-with-free-averages}, we may choose $\bdd{\psi} \in \mathbbb{G}_{\Gamma}$ in \cref{eq:stokes-system-gen-morgan-scott-1} such that $\rot \bdd{\psi} = r_{X}$ and $(\unitvec{t}\cdot \bdd{\psi}, 1)_{\Gamma_f^{(i)}} = \kappa_i$. Thus, $r_{X} \equiv 0$ and $\vec{\kappa} = \vec{0}$. Consequently, solutions to \cref{eq:stokes-system-gen-morgan-scott} are unique and hence exist since the spaces $\mathbbb{G}_{\Gamma}$, $\rot \mathbbb{G}_{\Gamma}$, and $\mathbb{R}^{N-1}$ are finite dimensional. Hence, the data for \cref{eq:morgan-scott-final-projection} vanishes and we conclude that $\tilde{w}_X \equiv 0$.
\end{proof}

\noindent In the remainder of this section, we apply \cref{thm:morgan-scott-simplified-ext} to various common choices of $H^2$-conforming finite element spaces.

\subsection{Morgan-Scott elements}
\label{sec:morgan-scott}

We use \cref{thm:morgan-scott-simplified-ext} to extend the result for the Morgan-Scott element in \cref{lem:morgan-scott-simplified} to the case $\Gamma_{cs}$ is not connected ($N > 1$):
\begin{lemma}
	\label{lem:vp-specify-constants}
	Let $\mathbb{W} = W^p$ \cref{eq:wp-morgan-scott}, $\mathbbb{G}_{\Gamma} = \bdd{G}_{\Gamma}^{p-1}$ \cref{eq:thetap-morgan-scott}, and $\tilde{\mathbb{W}} = \tilde{W}_{\Gamma}^p$ \cref{eq:tildewp-morgan-scott}. If $p \geq 5$, then \ref{hp:wp-constants-boundary-assumption}, \ref{hp:tildewp-condition}, and \ref{hp:stokes-complex-bcs-discon-fem} hold. 
\end{lemma}
\begin{proof}
	By definition, $\grad W_{\Gamma}^p \subset \bdd{G}_{\Gamma}^{p-1}$. Now let $\{ \mu_i \}_{i=1}^{N} \subset \mathbb{R}$, $N \geq 1$, be given. We define a function belonging to the Morgan-Scott space of order 5 by setting all function value degrees of freedom \cite{MorganScott75} corresponding to points located on $\Gamma_{cs}^{(i)}$ to $\mu_i$, $1 \leq i \leq N$, and setting the remaining degrees of freedom in \cite{MorganScott75} to zero. The resulting function $w \in W^5$ satisfies $w|_{\Gamma_{cs}^{(i)}} = \mu_i$, $1 \leq i \leq N$, $\partial_n w|_{\Gamma} = 0$, and $\grad w \in \bdd{G}_{\Gamma}^{4}$. Thus, \ref{hp:wp-constants-boundary-assumption} holds.
	
	Condition \ref{hp:tildewp-condition} follows by definition. We now turn to the exactness of \cref{eq:stokes-complex-bcs-discon-fem}. Let $\bdd{\theta} \in  \bdd{G}^{p-1}_{\Gamma}$ satisfy $\rot \bdd{\theta} \equiv 0$. Since $\Omega$ is simply-connected, $\bdd{\theta} = \grad w$ for some $w \in H^1(\Omega)$ by \cite[Theorem 3.1]{GiraultRaviart86}. Since $\bdd{G}^{p-1}_{\Gamma} \subset \bdd{H}^1(\Omega)$, we have
	\begin{align*}
		w \in \{ u \in H^2(\Omega) : u|_{K} \in \mathcal{P}_p(K) \ \forall K \in \mathcal{T}, \ u|_{\Gamma_{cs}^{(i)}} \in \mathbb{R}, \ 1 \leq i \leq N \}.	
	\end{align*}
	Consequently, the function $v = w - w|_{\Gamma_{cs}^{(1)}}$ satisfies $v \in V_{\Gamma}^p$ and $\grad v = \bdd{\theta}$. The exactness of \cref{eq:stokes-complex-bcs-discon-fem} now follows.
\end{proof}
In view of \cref{lem:vp-specify-constants}, \cref{thm:morgan-scott-simplified-ext} shows that \cref{alg:abstract-method} produces the $H^2$-conforming Morgan-Scott approximation defined by \cref{eq:biharmonic-primal-fem} whenever $p \geq 5$.

\begin{remark}
	\label{rem:morgan-scott-low-order}
	\Cref{lem:vp-specify-constants} also holds for $1 \leq p \leq 4$ when $\Gamma_{cs}$ is connected ($N = 1$). In particular, $\grad W_{\Gamma}^p \subset \bdd{G}_{\Gamma}^{p-1}$ holds for any $p \geq 1$, and, given $\mu \in \mathbb{R}$, the function $w=\mu \in W^p$ satisfies \ref{hp:wp-constants-boundary-assumption}. Moreover, 	\ref{hp:tildewp-condition} follows by definition, and the same arguments in the proof of \cref{lem:vp-specify-constants} show that \ref{hp:stokes-complex-bcs-discon-fem} holds. This means that \cref{alg:abstract-method} can be used to compute the $H^2$-conforming Morgan-Scott approximation in the case $p < 5$. An example is presented in \cref{sec:low-order-example}.
\end{remark}

\subsection{Argyris and TUBA elements}
\label{sec:argyris}

The classic $H^2$-conforming element is the Argyris element \cite{Argyris68} which consists of degree 5 polynomials that have $C^2$ degrees of freedom element vertices. This is the lowest order element in the TUBA family that extends to polynomials of degree $p \geq 5$ as follows:
\begin{subequations}
	\label{eq:wp-argyris}
	\begin{align}
		A^p &:= \{ v \in C^1(\Omega) : v|_{K} \in \mathcal{P}_{p}(K) \ \forall K \in \mathcal{T}, \text{ $v$ is $C^2$ at element vertices} \}, \\
		A_{\Gamma}^p &:= A^p \cap H^2_{\Gamma}(\Omega).
	\end{align}
\end{subequations}
If we choose $\mathbb{W} = A^p$, the extra $C^2$-smoothness at the element vertices is inherited by the gradients of functions in $A_{\Gamma}^p$ which suggests choosing the corresponding gradient space $\mathbbb{G}_{\Gamma}$ to be the Hermite finite element space:
\begin{align}
	\label{eq:thetap-argyris}
	\bdd{H}^{p-1}_{\Gamma} := \{ \bdd{\theta} \in \bdd{\Theta}_{\Gamma}(\Omega) : \bdd{\theta}|_{K} \in [\mathcal{P}_{p-1}(K)]^2 \ \forall K \in \mathcal{T}, \text{  $\bdd{\theta}$ is $\bdd{C}^1$ at element vertices} \}.
\end{align}  
As before, we choose $\tilde{\mathbb{W}}_{\Gamma}$ to be the $H^1$-conforming space $\tilde{W}_{\Gamma}^p$ defined in \cref{eq:tildewp-morgan-scott}. The following lemma shows that this choice satisfies \ref{hp:wp-constants-boundary-assumption}, \ref{hp:tildewp-condition}, and \ref{hp:stokes-complex-bcs-discon-fem}:
\begin{lemma}
	\label{lem:argyris-satis-conds13}
	The spaces $\mathbb{W} = A^p$, $\mathbbb{G}_{\Gamma} = \bdd{H}_{\Gamma}^{p-1}$, and $\tilde{\mathbb{W}}_{\Gamma} = \tilde{W}_{\Gamma}^p$, $p \geq 5$,  satisfy \ref{hp:wp-constants-boundary-assumption}, \ref{hp:tildewp-condition}, and \ref{hp:stokes-complex-bcs-discon-fem}.
\end{lemma}
\begin{proof}
	$\grad \bdd{A}_{\Gamma}^p \subset \bdd{H}_{\Gamma}^{p-1}$ and $A_{\Gamma}^p \subset \tilde{W}_{\Gamma}^p$ by definition, which verifies \ref{hp:tildewp-condition}.  Now let $\{ \mu_i \}_{i=1}^{N} \subset \mathbb{R}$, $N \geq 1$, be given. We define a function belonging to the Argyris space by setting all function value degrees of freedom \cite{Argyris68} corresponding to points located on $\Gamma_{cs}^{(i)}$ to $\mu_i$, $1 \leq i \leq N$, and setting the remaining degrees of freedom in \cite{Argyris68} to zero. The resulting function $w \in A^5$ satisfies $w|_{\Gamma_{cs}^{(i)}} = \mu_i$, $1 \leq i \leq N$, $\partial_n w|_{\Gamma} = 0$, and $\grad w \in \bdd{H}_{\Gamma}^{4}$. Thus, \ref{hp:wp-constants-boundary-assumption} holds.
	
	Now let $\bdd{\theta} \in  \bdd{G}^{p-1}_{\Gamma}$ satisfy $\rot \bdd{\theta} \equiv 0$. Arguing as in the proof of \cref{lem:vp-specify-constants}, we have that $\bdd{\theta} = \grad w$ for some $w \in H^1(\Omega)$. Since $\bdd{H}^{p-1}_{\Gamma} \subset \bdd{H}^1(\Omega)$, there holds
	\begin{multline*}
		w \in \{ u \in H^2(\Omega) : u|_{K} \in \mathcal{P}_p(K) \ \forall K \in \mathcal{T}, \ u \text{ is $C^2$ at element vertices and} \\
		u|_{\Gamma_{cs}^{(i)}} \in \mathbb{R}, \ 1 \leq i \leq N \}.	
	\end{multline*}
	Consequently, the function $v = w - w|_{\Gamma_{cs}^{(1)}}$ satisfies $v \in \mathbb{V}_{\Gamma}$ and $\grad v = \bdd{\theta}$, and so \ref{hp:stokes-complex-bcs-discon-fem} follows.
\end{proof}
\Cref{thm:morgan-scott-simplified-ext} shows that \cref{alg:abstract-method} can be used to compute the $H^2$-conforming Argyris/TUBA approximation defined by \cref{eq:biharmonic-primal-fem}.

\subsection{Hsieh-Clough-Tocher (HCT) element}
\label{sec:hct}

The archetypal low-order $H^2$-conforming element $(p=3)$ is the HCT macroelement \cite{Clough65} defined as follows:
\begin{align}
	\label{eq:wp-hct}
	HCT &:= \{ v \in C^1(\Omega) : v|_{K} \in \mathcal{P}_{3}(K) \ \forall K \in \mathcal{T}^{*} \},
\end{align}
where $\mathcal{T}^{*}$ denotes the mesh obtained from performing a barycentric refinement on every triangle in $\mathcal{T}$. The corresponding gradient space is chosen to be $\mathbbb{G}_{\Gamma} = \{ \bdd{\theta} \in \bdd{\Theta}_{\Gamma}(\Omega) : \bdd{\theta}|_{K} \in [\mathcal{P}_{2}(K)]^2 \ \forall K \in \mathcal{T}^{*} \}$. while $\tilde{\mathbb{W}}_{\Gamma}$ is chosen to be $\tilde{\mathbb{W}}_{\Gamma} = \{ v \in C(\Omega) : v|_{K} \in \mathcal{P}_3(K) \ \forall K \in \mathcal{T}^{*} \text{ and } v|_{\Gamma_{cs}} = 0 \}$. One may verify that taking $\mathbb{W} = HCT$ with the above choices satisfies \ref{hp:wp-constants-boundary-assumption}, \ref{hp:tildewp-condition}, and \ref{hp:stokes-complex-bcs-discon-fem} using the similar arguments as in the previous two sections -- we omit the details. Once again, \cref{thm:morgan-scott-simplified-ext} means that \cref{alg:abstract-method} can be used to compute the $H^2$-conforming HCT approximation of \cref{eq:biharmonic-primal-fem}.

\section{Implementation and numerical examples}
\label{sec:implementation-numerics}

Let $\mathbb{W}$, $\mathbbb{G}_{\Gamma}$, and $\tilde{\mathbb{W}}_{\Gamma}$ be given spaces satisfying \ref{hp:wp-constants-boundary-assumption}, \ref{hp:tildewp-condition}, and \ref{hp:stokes-complex-bcs-discon-fem}. One remaining obstacle to computing the solution to \cref{eq:zp-morgan-scott-2} is the appearance of the space $\rot \mathbbb{G}_{\Gamma}$. In principle, one can construct a basis for this finite dimensional space and formulate \cref{eq:zp-morgan-scott-2} as a linear algebraic system. However, it is not straightforward to find an explicit basis for this type of space \cite{AinCP22SCIP,AinCP21LE,ScottVog85}. 
 
An alternative is to use the iterated penalty method \cite{Brenner08,FortinGlow83,Glowinski84} which enables one to compute the solution to \cref{eq:stokes-system-gen-morgan-scott} without the need for a basis for the space $\rot \mathbbb{G}_{\Gamma}$. Specifically we define the following bilinear forms on $\mathbbb{G}_{\Gamma} \times \mathbbb{G}_{\Gamma}$ for a given $\lambda \geq 0$:
\begin{subequations}
	\begin{align}
		b(\bdd{\theta}, \bdd{\psi}) &:= (\rot \bdd{\theta}, \rot \bdd{\psi}) + \sum_{i=1}^{N-1} (\unitvec{t} \cdot \bdd{\theta}, 1)_{\Gamma_f^{(i)}} (\unitvec{t} \cdot \bdd{\psi}, 1)_{\Gamma_f^{(i)}}, \\
		a_{\lambda}(\bdd{\theta}, \bdd{\psi}) &:= a(\bdd{\theta}, \bdd{\psi})  + \lambda b(\bdd{\theta}, \bdd{\psi}).
	\end{align}
\end{subequations}
Given $\lambda \geq 0$ and $\bdd{\phi}^0 \in \mathbbb{G}_{\Gamma}$, the iterated penalty method for \cref{eq:stokes-system-gen-morgan-scott} is as follows: For $n = 0,1,\ldots$, define $ \bdd{\theta}^n, \bdd{\phi}^{n+1} \in \mathbbb{G}_{\Gamma}$ by
\begin{subequations}
	\label{eq:iter-penalty-morgan-scott}
	\begin{alignat}{2}
		a_{\lambda}(\bdd{\theta}^{n}, \bdd{\psi})   &= (\grad \tilde{z}_X, \bdd{\psi}) - b(\bdd{\phi}^n, \bdd{\psi}) \qquad & & \forall \bdd{\psi} \in \mathbbb{G}_{\Gamma}, \\
		\bdd{\phi}^{n+1} &= \bdd{\phi}^{n} + \lambda \bdd{\theta}^n. \qquad & & 
	\end{alignat}
\end{subequations}
Under appropriate conditions, given below, the iterates converge to the discrete solution to \cref{eq:stokes-system-gen-morgan-scott} as $n\to \infty$: 
\begin{align*}
	\bdd{\theta}^n \to \bdd{\theta}_{X}, \quad \rot \bdd{\phi}^n \to r_{X}, \quad \text{and} \quad (\unitvec{t}\cdot \bdd{\phi}^n, 1)_{\Gamma_f^{(i)}} \to \kappa_i, \quad 1 \leq i \leq N-1.
\end{align*}
To quantify the rate of convergence, we let $B$ denote a positive constant satisfying
\begin{align*}
	|b(\bdd{\theta}, \bdd{\psi})| \leq B \| \bdd{\theta}\|_1 \|\bdd{\psi}\|_1 \qquad \bdd{\theta}, \bdd{\psi} \in \bdd{\Theta}_{\Gamma}.
\end{align*}
 The existence of such constant follows from the trace theorem. The effect of the choice of finite elements spaces on the rate of convergence is measured by a constant $\beta_X$ defined as follows:
\begin{align}
	\label{eq:stokes-morgan-scott-gen-inf-sup}
	\beta_X := \inf_{ \substack{ (r, \vec{\kappa}) \in \rot \mathbbb{G}_{\Gamma} \times \mathbb{R}^{N-1} \\ (r, \vec{\kappa}) \neq (0, \vec{0}) } } \sup_{ \substack{ \bdd{\theta} \in \mathbbb{G}_{\Gamma} \\ \bdd{\theta} \neq \bdd{0} } } \frac{ (\rot \bdd{\theta}, r) + \sum_{i=1}^{N-1} (\unitvec{t} \cdot \bdd{\theta}, \kappa_i)_{\Gamma_f^{(i)}} }{ (\|r\| + |\vec{\kappa}|) \|\bdd{\theta}\|_{1} },
\end{align} 
where $|\vec{\kappa}| := \sum_{i=1}^{N-1} |\kappa_i|$. For each $(r, \vec{\kappa}) \in \rot \mathbbb{G}_{\Gamma} \times \mathbb{R}^{N-1}$, applying \cref{lem:invert-rot-with-free-averages} with $\omega_i = \kappa_i$, $1 \leq i \leq N-1$, and $\omega_N$ chosen so that \cref{eq:kappaij-constraint} holds shows that there exists a $\bdd{\theta} \in \mathbbb{G}_{\Gamma}$ for which the numerator is strictly positive. As a consequence. the parameter $\beta_X$ is strictly positive, but may depend on $N$ and the dimension of $\mathbbb{G}_{\Gamma}$. The rate of convergence of the iterated penalty scheme \cref{eq:iter-penalty-morgan-scott} depends on the above constants as follows:
\begin{lemma}
	\label{lem:iter-penalty-morgan-scott}
	Let $(\bdd{\theta}_{X}, r_{X}, \vec{\kappa}) \in \mathbbb{G}_{\Gamma} \times \rot \mathbbb{G}_{\Gamma} \times \mathbb{R}^{N-1}$ satisfy \cref{eq:stokes-system-gen-morgan-scott}. Then, for all
	\begin{align*}
		\lambda > \lambda_0 := \frac{M}{\beta_X^2} \left( 1 + \frac{M}{\alpha} \right)^2,
	\end{align*}
	 the iterates $\bdd{\theta}^n, \bdd{\phi}^{n+1} \in \mathbbb{G}_{\Gamma}$ given by \cref{eq:iter-penalty-morgan-scott} are well-defined and satisfy:
	\begin{subequations}
		\label{eq:iter-penalty-morgan-scott-convergence-stokes}
		\begin{align}
				\|\bdd{\theta}^n - \bdd{\theta}_{X} \|_1 &\leq \frac{1}{\beta_X} \left( 1 + \frac{M}{\alpha} \right) \epsilon_n \\
				\| \rot \bdd{\phi}^n - r_{X}\| + \sum_{i=1}^{N-1} |(\unitvec{t} \cdot \bdd{\phi}^n, 1)_{\Gamma_f^{(i)}} - \kappa_i| &\leq \left( \frac{M}{\beta_X} \left( 1 + \frac{M}{\alpha} \right) + \lambda B \right) \epsilon_n,
		\end{align}
	\end{subequations}
	where
	\begin{align}
		\label{eq:iter-penalty-morgan-scott-rot}
		\epsilon_n := \|\rot \bdd{\theta}^n\| + \sum_{i=1}^{N-1} |(\unitvec{t} \cdot \bdd{\theta}^n, 1)_{\Gamma_f^{(i)}}| \leq B \left( \frac{\lambda_0}{\lambda} \right)^n \| \bdd{\theta}^0 - \bdd{\theta}_{X} \|_1.
	\end{align}	
\end{lemma}
\begin{proof}[Proof of \cref{lem:iter-penalty-morgan-scott}]
	Let $\lambda > \lambda_0 > 0$. To show that the iterates defined by \cref{eq:iter-penalty-morgan-scott} are well-defined, it suffices to show that if $\bdd{\theta} \in \mathbbb{G}_{\Gamma}$ satisfies $a_{\lambda}(\bdd{\theta}, \bdd{\psi}) = 0$ for all $\bdd{\psi} \in \mathbbb{G}_{\Gamma}$, then $\bdd{\theta} \equiv \bdd{0}$. Choosing $\bdd{\psi} = \bdd{\theta}$ and using the positivity of $a(\cdot,\cdot)$ \cref{eq:a-positive} gives
	\begin{align*}
		\| \rot \bdd{\theta}\|^2 + \sum_{i=1}^{N-1} \left|(\unitvec{t} \cdot \bdd{\theta}, 1)_{\Gamma_f^{(i)}}\right|^2 \leq \lambda^{-1} a_{\lambda}( \bdd{\theta}, \bdd{\theta} ) = 0.
	\end{align*}
	Arguing as in the proof of \cref{thm:morgan-scott-simplified-ext}, it follows that $\bdd{\theta} = \grad w$ for some $w \in \mathbb{W}_{\Gamma}$ and so $0 = a_{\lambda}(\grad w, \grad w) = a(\grad w, \grad w) \geq \alpha \|w\|_2^2$ by the coercivity of $a(\cdot, \cdot)$ on $\grad H^2_{\Gamma}$ \cref{eq:a-elliptic}. Consequently $w \equiv 0$ and hence $\bdd{\theta} \equiv \bdd{0}$.	
	
	\Cref{eq:iter-penalty-morgan-scott-convergence-stokes,eq:iter-penalty-morgan-scott-rot} now follow from standard iterated penalty convergence results; see e.g. equation (13.1.18) and Theorem 13.2.2 in \cite{Brenner08}.
\end{proof}

\noindent \Cref{lem:iter-penalty-morgan-scott} shows that the iterated penalty method converges at a geometric rate provided that the parameter $\lambda$ is chosen to be sufficiently large while \cref{eq:iter-penalty-morgan-scott-rot} provides a numerical stopping criterion for the iteration. 

\begin{remark}
	\label{rem:morgan-scott-constant-iter}
	In the special case when $\mathbb{W} = W^p$, $p \geq 5$, is taken to be the Morgan-Scott space \cref{eq:wp-morgan-scott}, \cref{cor:stokes-morgan-scott-gen-inf-sup} shows that $\beta_{X}$ is independent of both the mesh size and polynomial degree $p$. As a result, the iterated penalty method converges at a rate independent of the mesh size and polynomial order.
\end{remark}

\subsection{Example 1: Simply-supported plate with point load}
\label{sec:lshape}

For the first example, we consider an L-shaped Kirchhoff plate $\Omega = (0, 1)^2 \setminus [0.5, 1]^2$ with simple supports $(\Gamma_s = \Gamma)$ subject to a point load at $(0.66,0.33)$. The material parameters in the bilinear form $a(\cdot,\cdot)$ \cref{eq:abilinear-def} are chosen to be $E = 1.4\mathrm{e}6$, $\nu = 0.3$, and $\tau = 0.01$. We take $\mathbb{W} = W^p$, $p \geq 5$, to be the Morgan-Scott space \cref{eq:wp-morgan-scott} on the mesh displayed in \cref{fig:lshape-mesh-and-solution} and apply \cref{alg:abstract-method} with the spaces chosen as in \cref{sec:morgan-scott}. We use the iterated penalty method \cref{eq:iter-penalty-morgan-scott} with $\lambda = 10^3$ to compute the solution to \cref{eq:stokes-system-gen-morgan-scott}. Iterations are terminated when $\epsilon_n < 10^{-10}$. The iterated penalty method converged in $4$ iterations for $p=5$ and $3$ iterations for $p=6,\ldots,10$, in agreement with \cref{rem:morgan-scott-constant-iter}.  \Cref{fig:lshape-mesh-and-solution} displays the displacement $\tilde{w}_X = w_X$ for the $p=10$ solution. In particular, one observes the well-known phenomenon whereby the upper-left part of the plate $(0, 0.5) \times (0.5, 1)$ is deflected in the \textit{opposite direction} to that in which the point load acts. Moreover, $\tilde{w}_X$  is $C^1$-conforming as seen in \cref{fig:lshape-gradient}.

\begin{figure}[htb]
	\centering
	\begin{subfigure}[b]{0.48\linewidth}
		\centering
		\includegraphics[width=\linewidth]{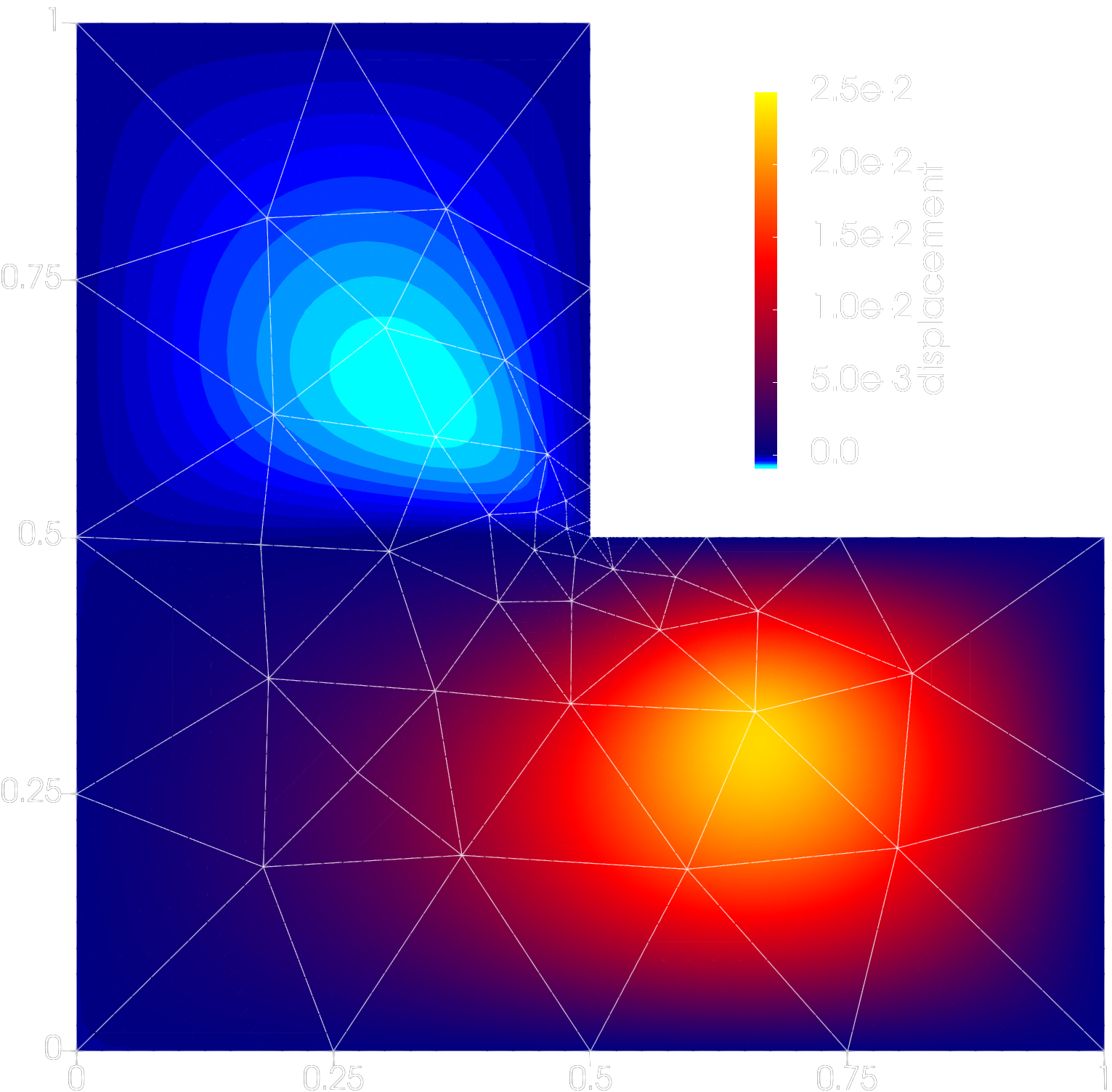}
		\caption{}
		\label{fig:lshape-mesh-and-solution}
	\end{subfigure}
	\hfill
	\begin{subfigure}[b]{0.48\linewidth}
		\centering
		\includegraphics[width=\linewidth]{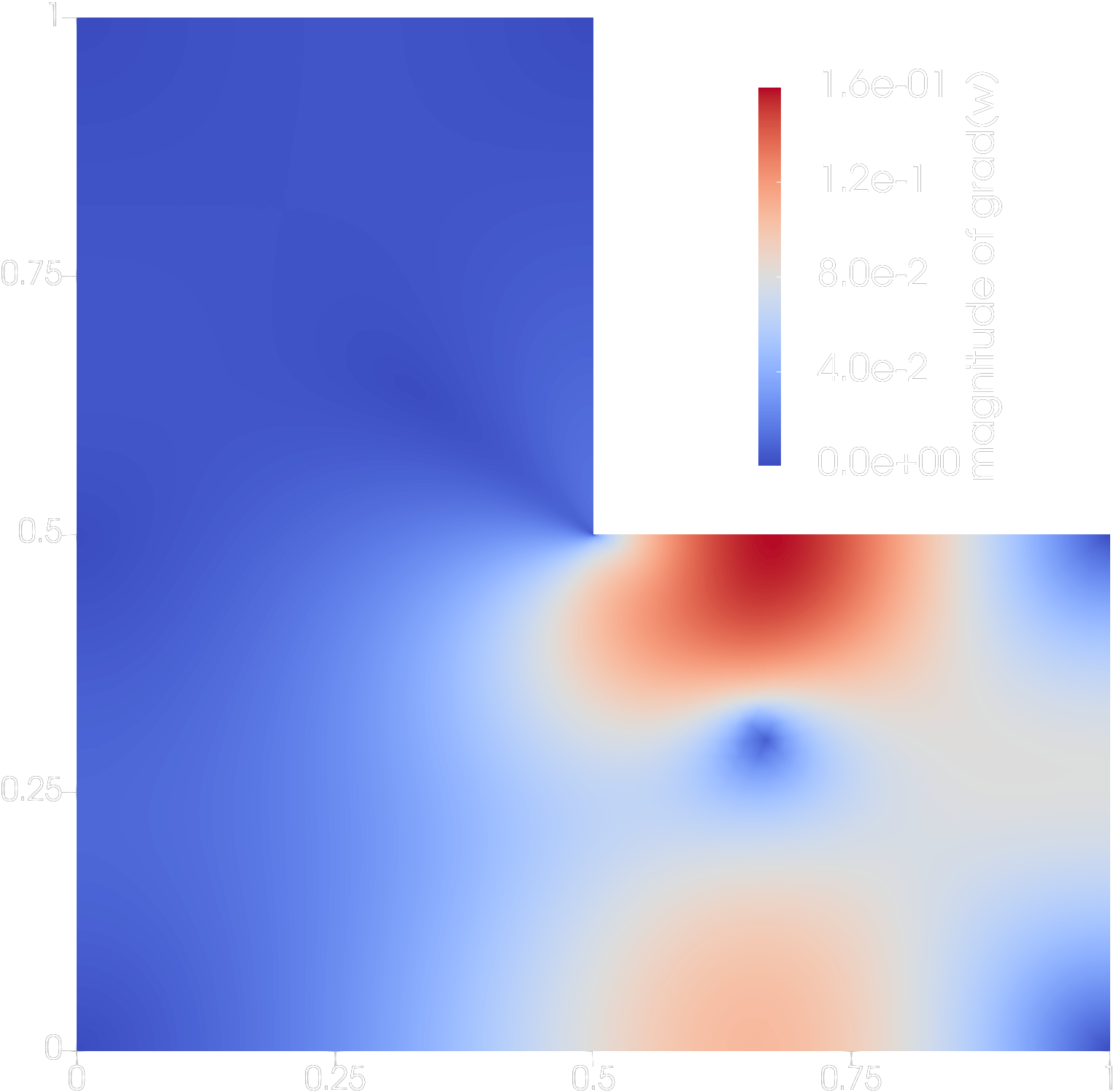}
		\caption{}
		\label{fig:lshape-gradient}
	\end{subfigure}
	\caption{(a) The $p=10$ Morgan-Scott approximation $\tilde{w}_X$ on the computational mesh (white) to the simply-supported L-shaped plate with point load  and (b) $|\grad \tilde{w}_X|$ (different color scale). Note that the upper left quadrant deflects in the opposite direction to which the load acts and that $\tilde{w}_X$ is indeed $C^1$-conforming.}
\end{figure}

One quantity of interest in elastic simulations is the von Mises stress \cite{vonmises13} 
\begin{align}
	\label{eq:von-mises}
	\sigma_v^2 := \sigma_{11}^2 + \sigma_{22}^2 - \sigma_{11} \sigma_{22} + 3 \sigma_{12}^2,
\end{align}
which is used to predict yielding. The value of $\sigma_v$ computed from the $p=10$ solution at the top of the plate is shown in \cref{fig:lshape-vonmises}. The true solution contains a singularity at the reentrant corner of the plate which leads to oscillations in the numerical approximation. However, these oscillations are confined to the elements that touch the reentrant corner and do not pollute the numerical solution away from the reentrant corner, as seen e.g. in the zoom on the point load in \cref{fig:lshape-vonmises-load}.

\begin{figure}[htb]
	\centering
	\begin{subfigure}{0.48\linewidth}
		\centering
		\includegraphics[width=\linewidth]{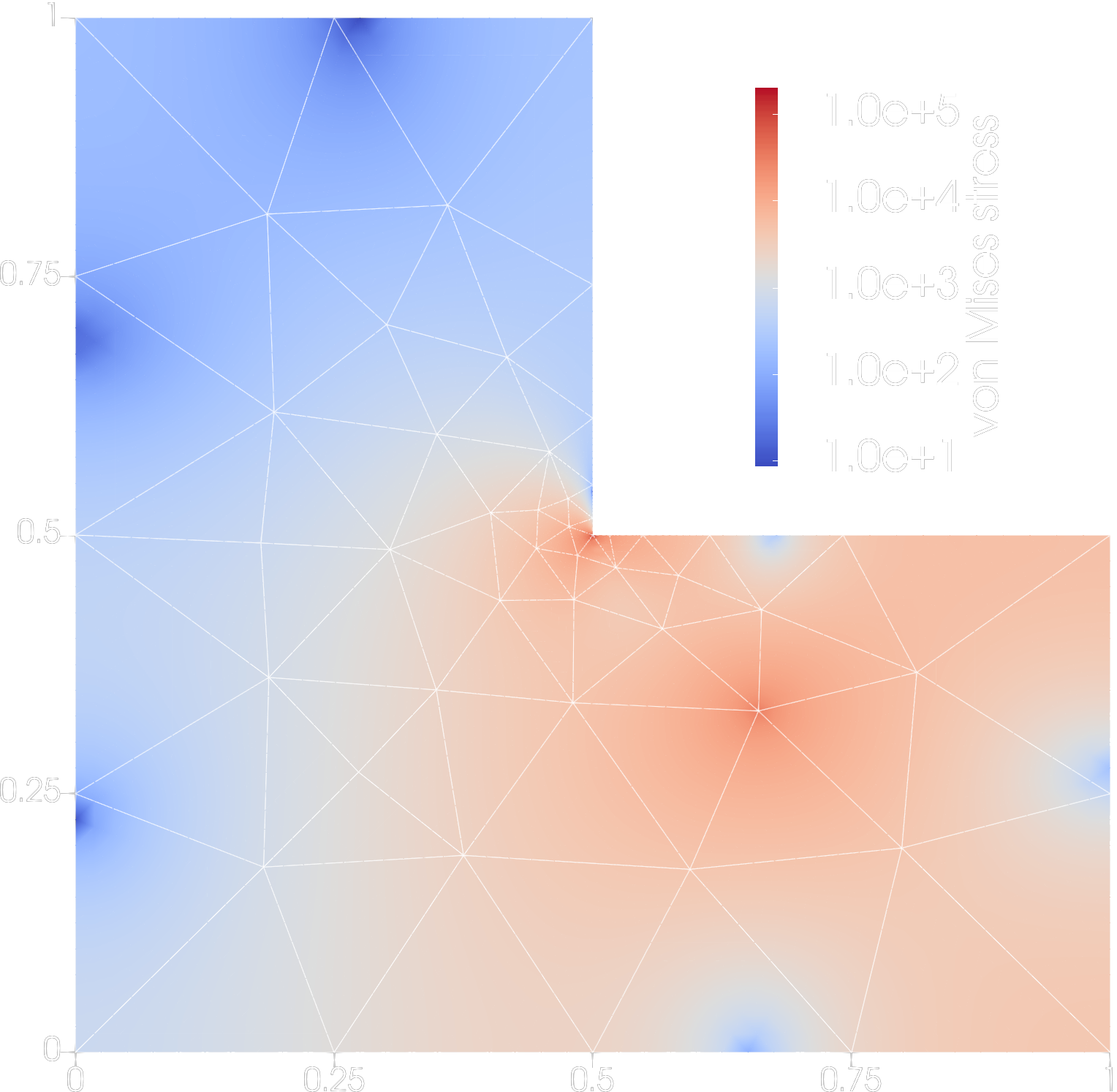}
		\caption{}
		\label{fig:lshape-vonmises}
	\end{subfigure}
	\hfill
	\begin{subfigure}{0.48\linewidth}
		\centering
		\includegraphics[width=\linewidth]{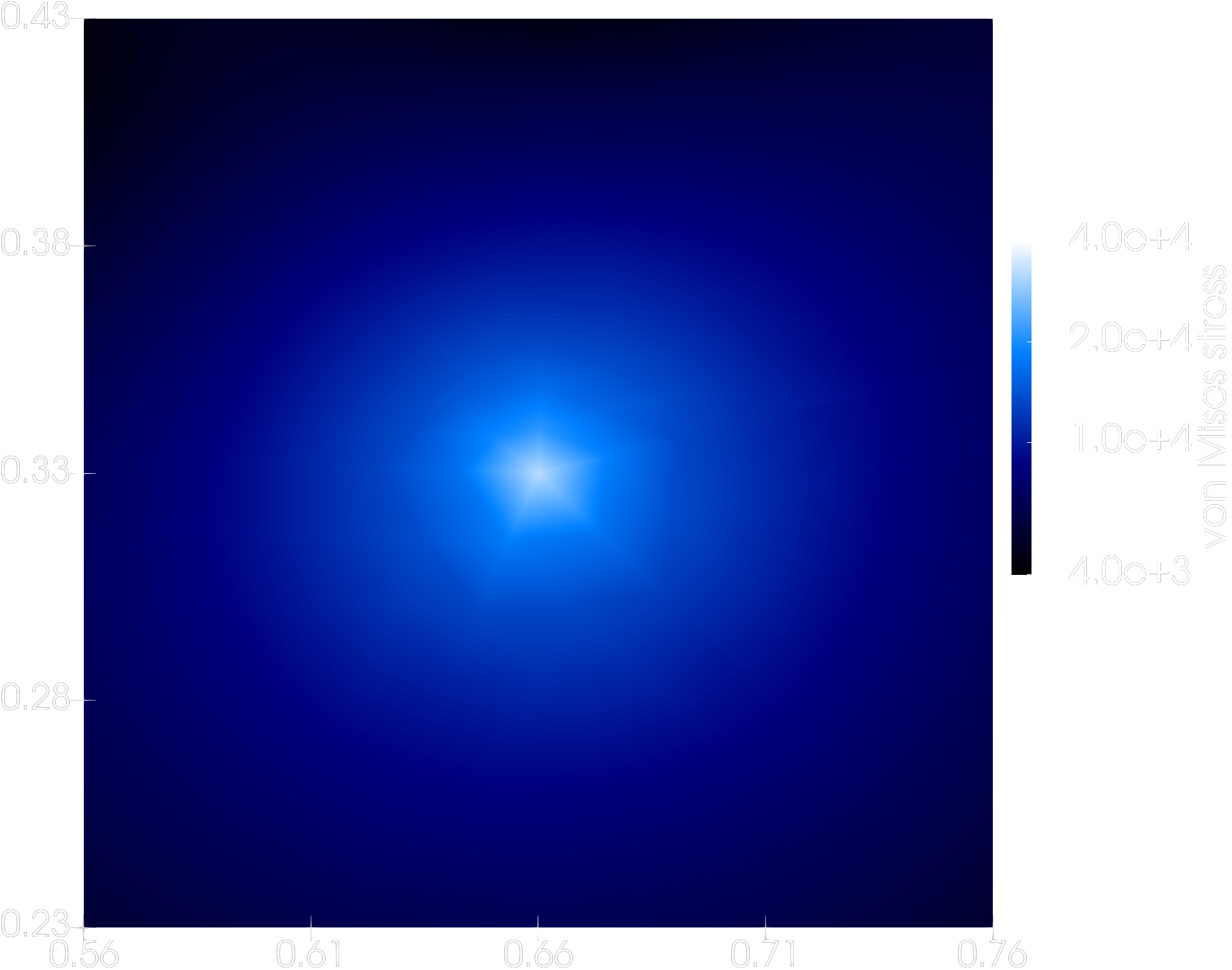}
		\caption{}
		\label{fig:lshape-vonmises-load}
	\end{subfigure}
	\caption{(a) The von Mises stress $\sigma_v$ \cref{eq:von-mises} computed from the $p=10$ Morgan-Scott approximation to the simply-supported L-shaped plate with point load and (b) zoom on the point load (different color scale).}
\end{figure}

\subsection{Example 2: G-shaped domain}
\label{sec:gshape}

We now turn to a G-shaped domain shown in \cref{fig:gshape-disp} subject to a uniform load $F \equiv 1$ with clamped vertical edges and free horizontal edges so that $\Gamma_{cs}$ consists of $N = 7$ connected components. We choose the material parameters, use the Morgan-Scott elements, and solve the Stokes-like system  \cref{eq:stokes-system-gen-morgan-scott} using the iterated penalty method \cref{eq:iter-penalty-morgan-scott} as in \cref{sec:lshape}. For $p=5,\ldots,10$, the iterated penalty method converged in 9 iterations. The $p=10$ approximation of the displacement and $|\grad \tilde{w}_X|$ are displayed in \cref{fig:gshape}. \Cref{alg:abstract-method} enables us to compute the Morgan-Scott solution and the iterated penalty method converges in a number of iterations independent of the polynomial degree in accordance with \cref{rem:morgan-scott-constant-iter}. The stresses are singular at every corner of the domain and, ordinarily, one would refine the mesh in the neighborhood of the corner to resolve the singularities. However, for illustrative purposes, no such refinements are carried out, and \cref{fig:gshape-vms} illustrates some pollution in the von Mises stress. Nevertheless, the computed approximation $\tilde{w}_{X}$ is seen to be $C^1$-conforming.

\begin{figure}[htb]
	\centering
	\begin{subfigure}[b]{0.48\linewidth}
		\centering
		\includegraphics[width=\linewidth]{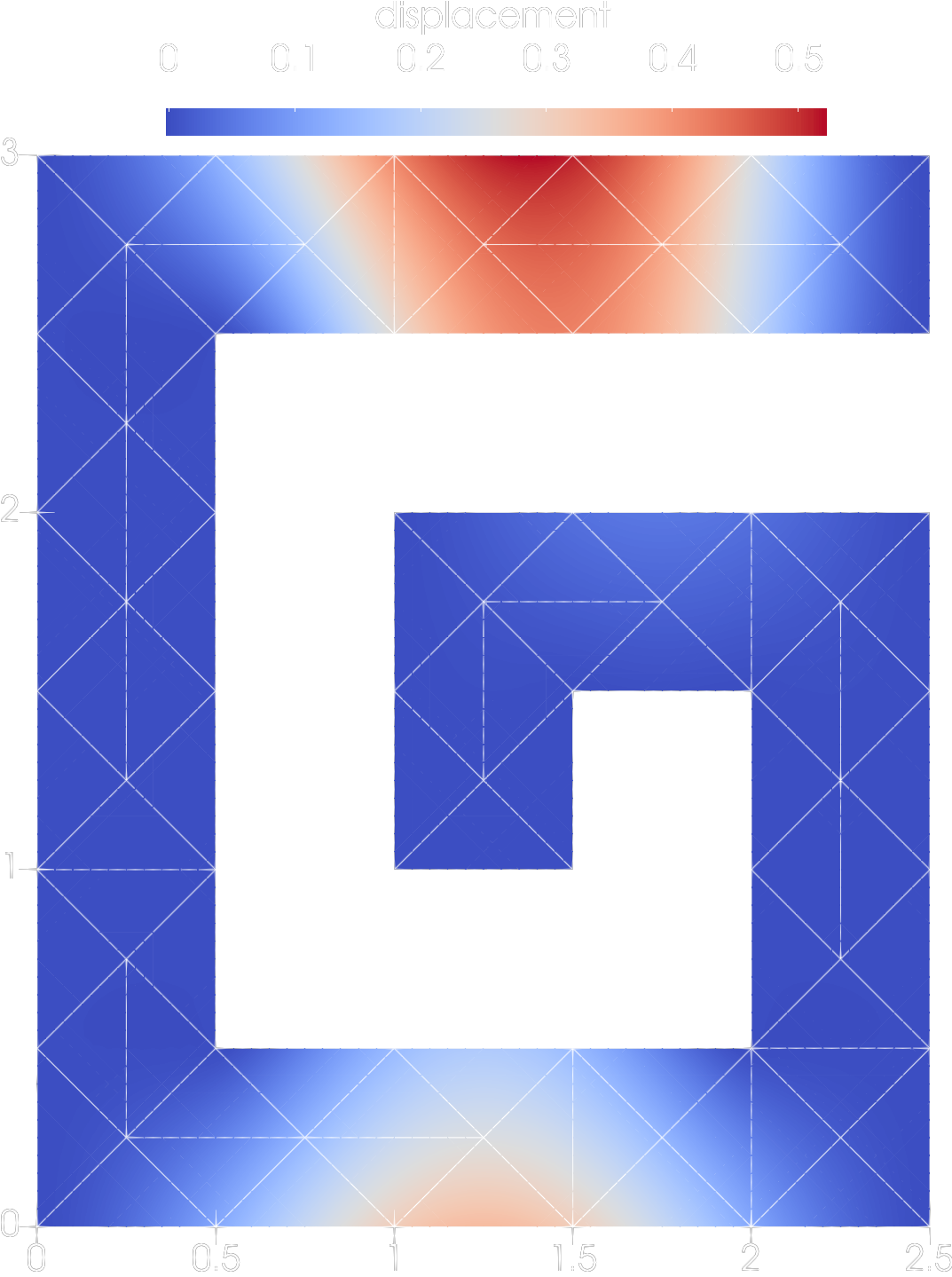}
		\caption{}
		\label{fig:gshape-disp}
	\end{subfigure}
	\hfill
	\begin{subfigure}[b]{0.48\linewidth}
		\centering
		\includegraphics[width=\linewidth]{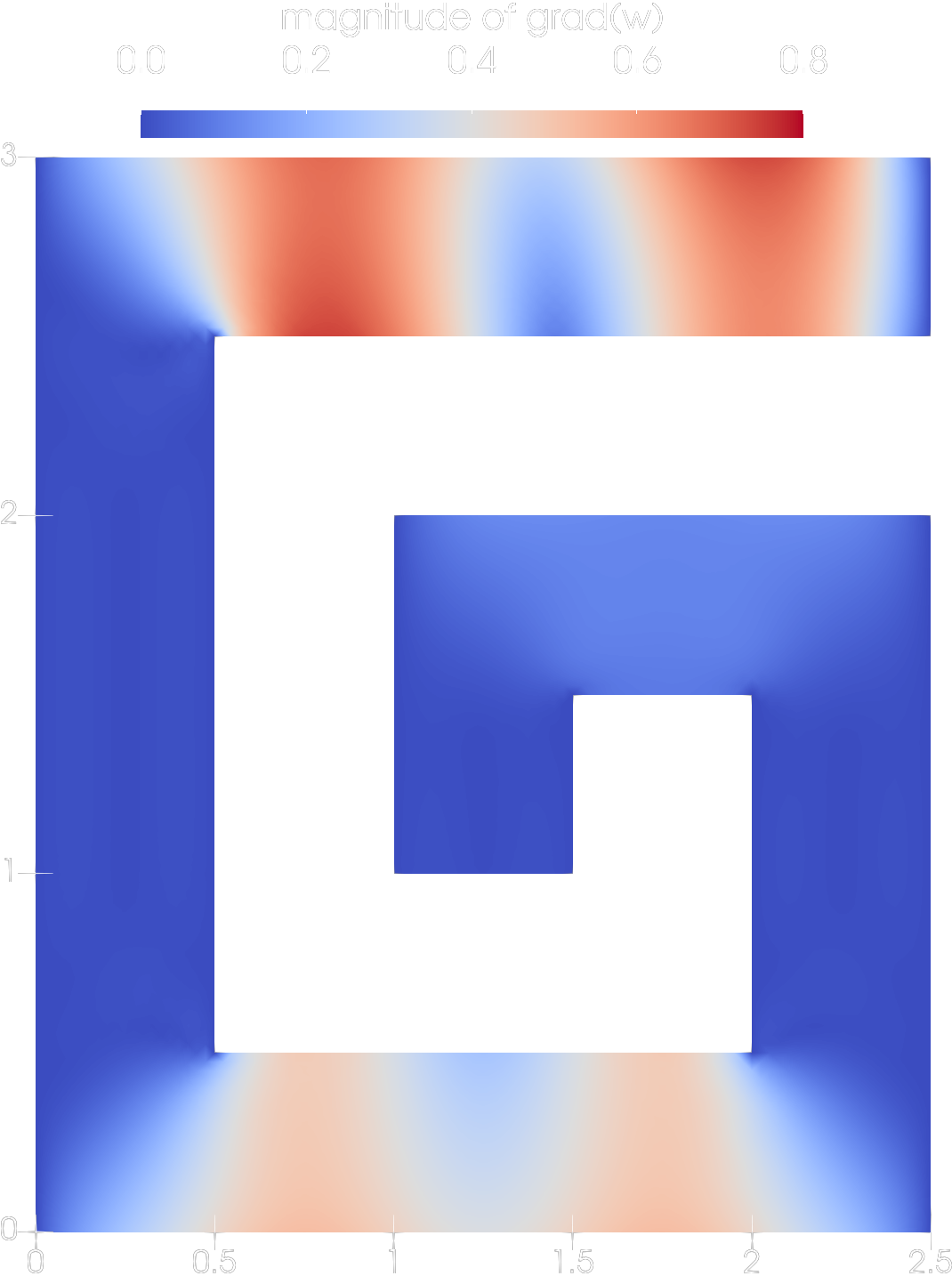}
		\caption{}
	\end{subfigure}
	\caption{(a) The $p=10$ Morgan-Scott approximation on the computational mesh (white) to the G-shaped plate with uniform load  and (b) $|\grad \tilde{w}_X|$ computed from the approximation (different color scale).}
	\label{fig:gshape}
\end{figure}

\begin{figure}[htb]
	\centering
	\begin{subfigure}[b]{0.48\linewidth}
		\centering
		\includegraphics[width=\linewidth]{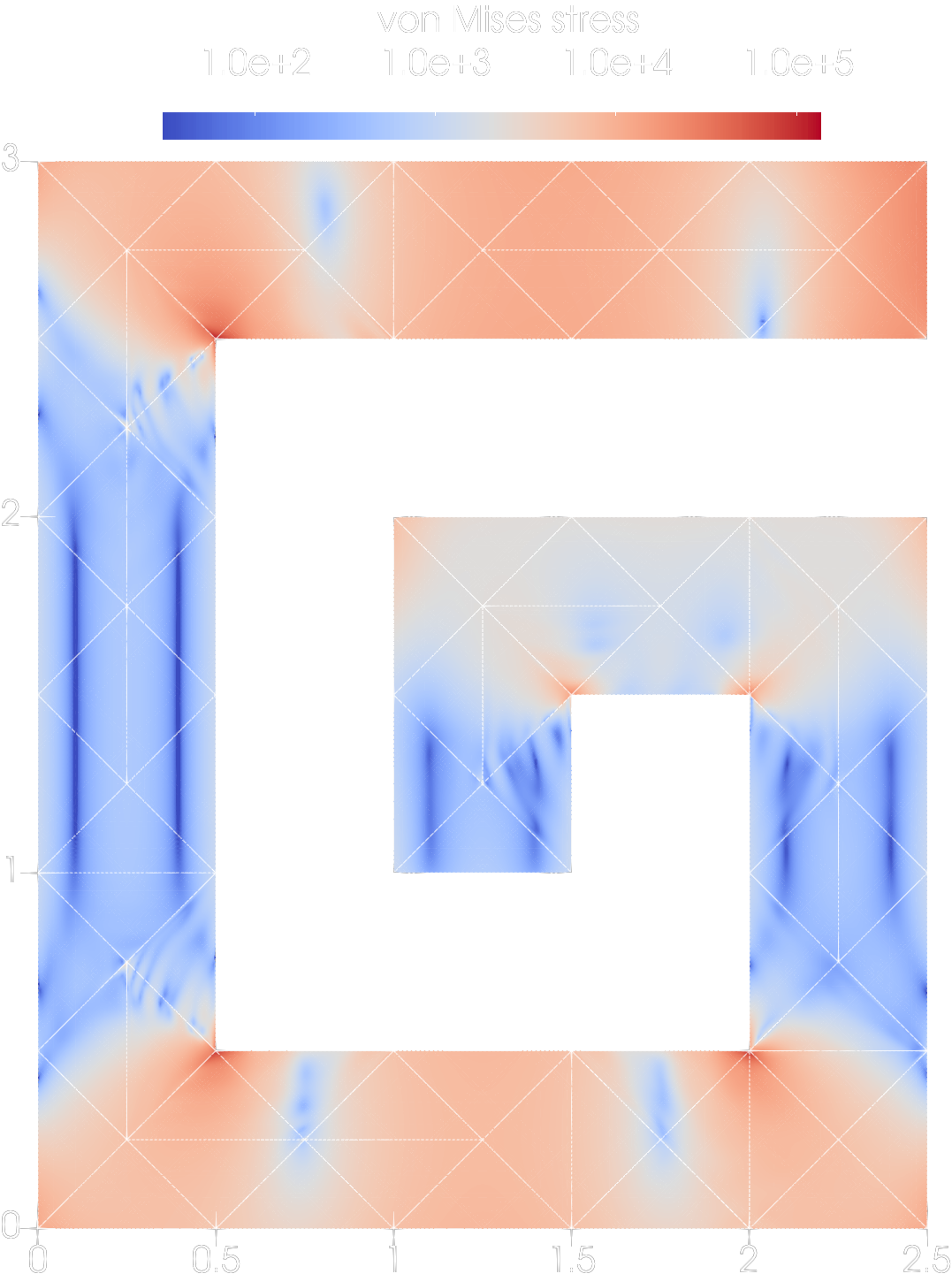}
		\caption{}
		\label{fig:gshape-vm}
	\end{subfigure}
	\hfill
	\begin{subfigure}[b]{0.48\linewidth}
		\centering
		\includegraphics[width=\linewidth]{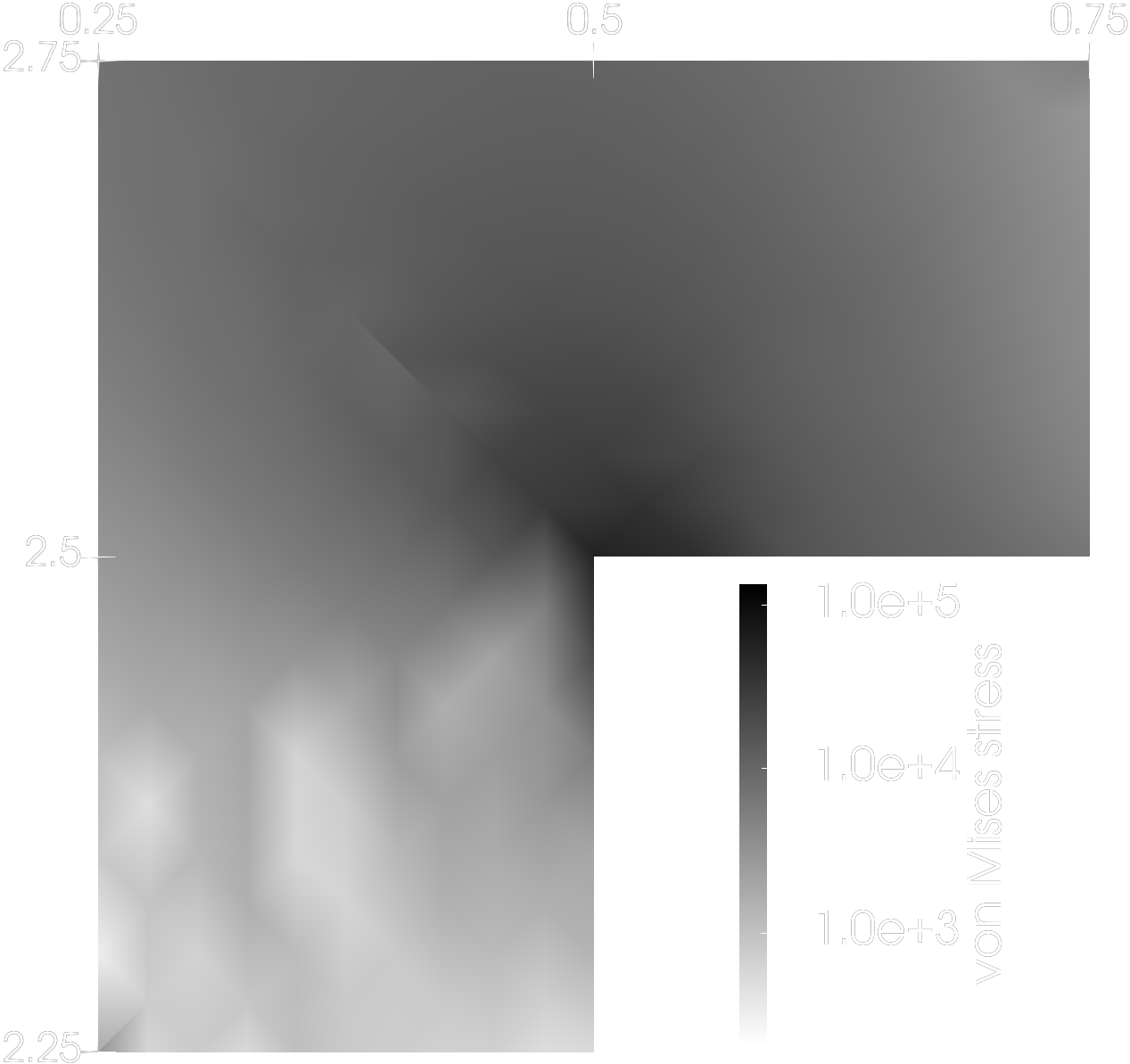}
		\caption{}
	\end{subfigure}
	\caption{(a) The von Mises stress $\sigma_v$ \cref{eq:von-mises} computed from the $p=10$ Morgan-Scott approximation to the G-shaped plate with uniform load  and (b) zoom on a reentrant corner (different color scale).}
	\label{fig:gshape-vms}
\end{figure}

\subsection{Example 3: Behavior of low-order methods}
\label{sec:low-order-example}

In the previous two examples, we chose $\mathbb{W}$ to be the Morgan-Scott space $W^p$ \cref{eq:wp-morgan-scott} of degree $p \geq 5$. As briefly mentioned earlier, the spaces $W^p$, $\mathbbb{G}_{\Gamma} = \bdd{G}_{\Gamma}^{p-1}$ \cref{eq:thetap-morgan-scott}, and  $\tilde{\mathbb{W}}_{\Gamma} = \tilde{W}_{\Gamma}^{p}$ \cref{eq:tildewp-morgan-scott} are well-defined for any $p \geq 1$, but there is no standard local, finite element basis for these elements \cite{AlfredPiperSchumaker87} and, as such it is not clear how one could compute the finite element approximation in the cases $p \geq 4$ using a standard element sub-assembly approach. In contrast, \cref{alg:abstract-method} does not require such a local basis and, as a result, means that the cases $p \leq 4$ can be treated in exactly the same way as $p \geq 5$. To highlight some approximation properties of the low-order ($p \leq 4$) Morgan-Scott elements, we solve the biharmonic projection problem
\begin{align}
	\label{eq:biharmonic-projection}
	w_p \in W_{\Gamma}^p : \qquad (\Delta w_p, \Delta v) = (\Delta^2 w, v) \qquad \forall v \in W_{\Gamma}^p,
\end{align}
where $\Gamma_c = \Gamma$ and $w := \sin^2(\pi x) \sin^2(\pi y)$. 

As shown in \cref{sec:problem-setting}, the bilinear form $a(\cdot,\cdot) := (\dive \cdot, \dive \cdot)$ satisfies \cref{eq:a-bounded-elliptic}. Moreover, \cref{rem:morgan-scott-low-order} shows that since $\Gamma_{cs}$ is connected ($N = 1$), $\mathbb{W}$, $\mathbbb{G}_{\Gamma}$, and $\tilde{\mathbb{W}}_{\Gamma}$ satisfy conditions \ref{hp:wp-constants-boundary-assumption}, \ref{hp:tildewp-condition}, and \ref{hp:stokes-complex-bcs-discon-fem}. Consequently, \cref{alg:abstract-method} delivers the lower-order approximations, and the iterated penalty method \cref{eq:iter-penalty-morgan-scott} is well-defined and converges at a geometric rate for $\lambda$ large enough, \textit{even though local bases for the spaces $W^p$, $p \leq 4$, are not available}. 

For each $2 \leq p \leq 5$, we compute the solution to \cref{eq:biharmonic-projection} using \cref{alg:abstract-method} on a sequence of meshes obtained by subdividing the unit square into two triangles and refining the mesh by splitting each triangle into four congruent subtriangles. The $H^2$-errors displayed in \cref{fig:low-order-convergence} show that for $2 \leq p \leq 4$, the errors behave like $\mathcal{O}(h^{p-2})$, which is consistent with the upper bound on the convergence rate of $p-2$ obtained in \cite{Deboor93}. Meanwhile, the $p=5$ approximation achieves the optimal error rate of $\mathcal{O}(h^4)$. Additionally, the number of iterations for the iterated penalty method to converge grows as the mesh is refined for $2 \leq p \leq 4$ but remains constant for $p=5$. This observation is consistent with the behavior of the inf-sup constant \cref{eq:stokes-morgan-scott-gen-inf-sup}, which behaves like $\mathcal{O}(h)$ for $2 \leq p \leq 4$ (see e.g. \cite[Figure 13]{AinCP21LE}) and is uniformly bounded from below for $p=5$ thanks to \cref{lem:invert-rot-rp}.

\begin{figure}[htb]
	\centering
	\begin{subfigure}[b]{0.48\linewidth}
		\centering
		\includegraphics[width=\linewidth]{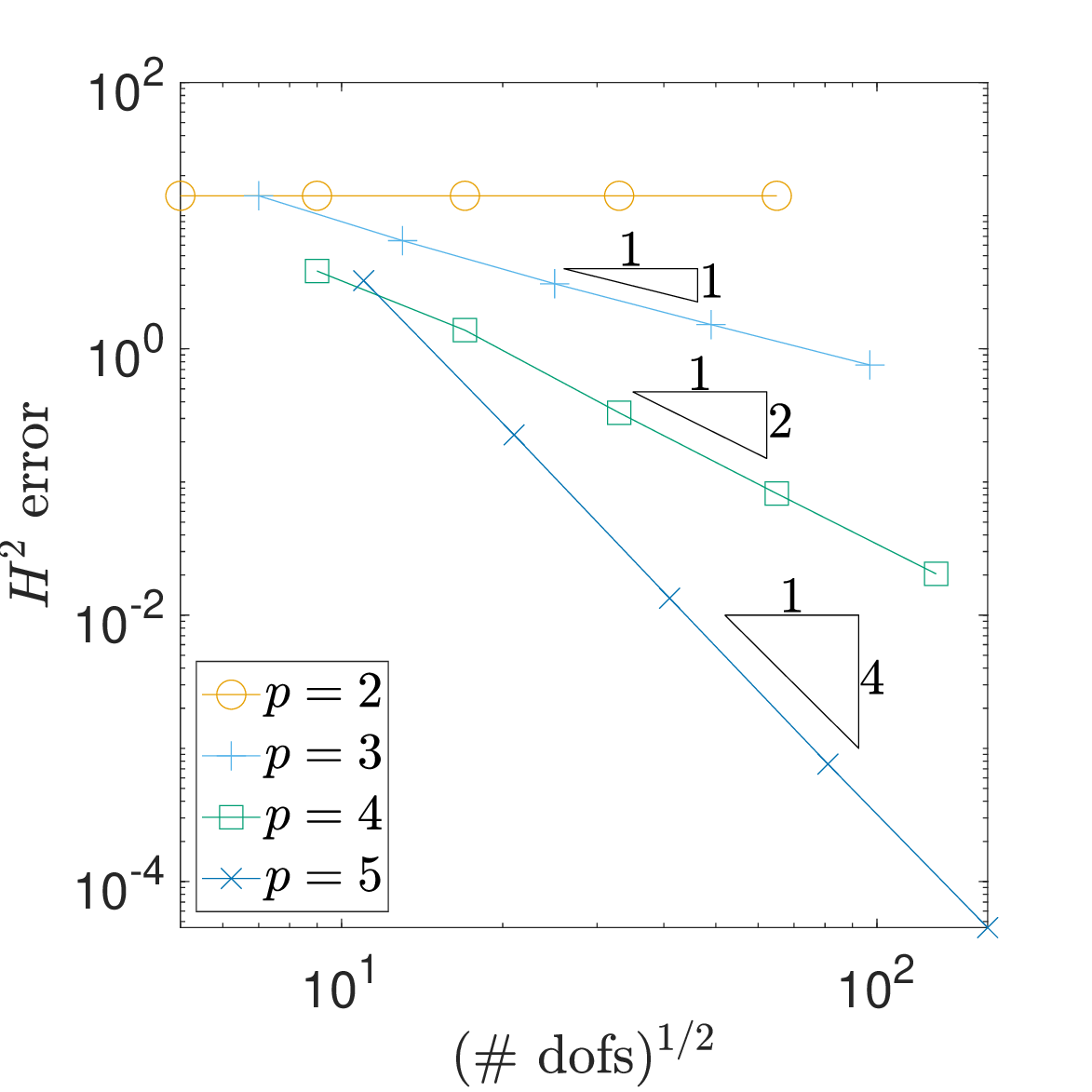}
		\caption{}
		\label{fig:low-order-convergence}
	\end{subfigure}
	\hfill
	\begin{subfigure}[b]{0.48\linewidth}
		\centering
		\includegraphics[width=\linewidth]{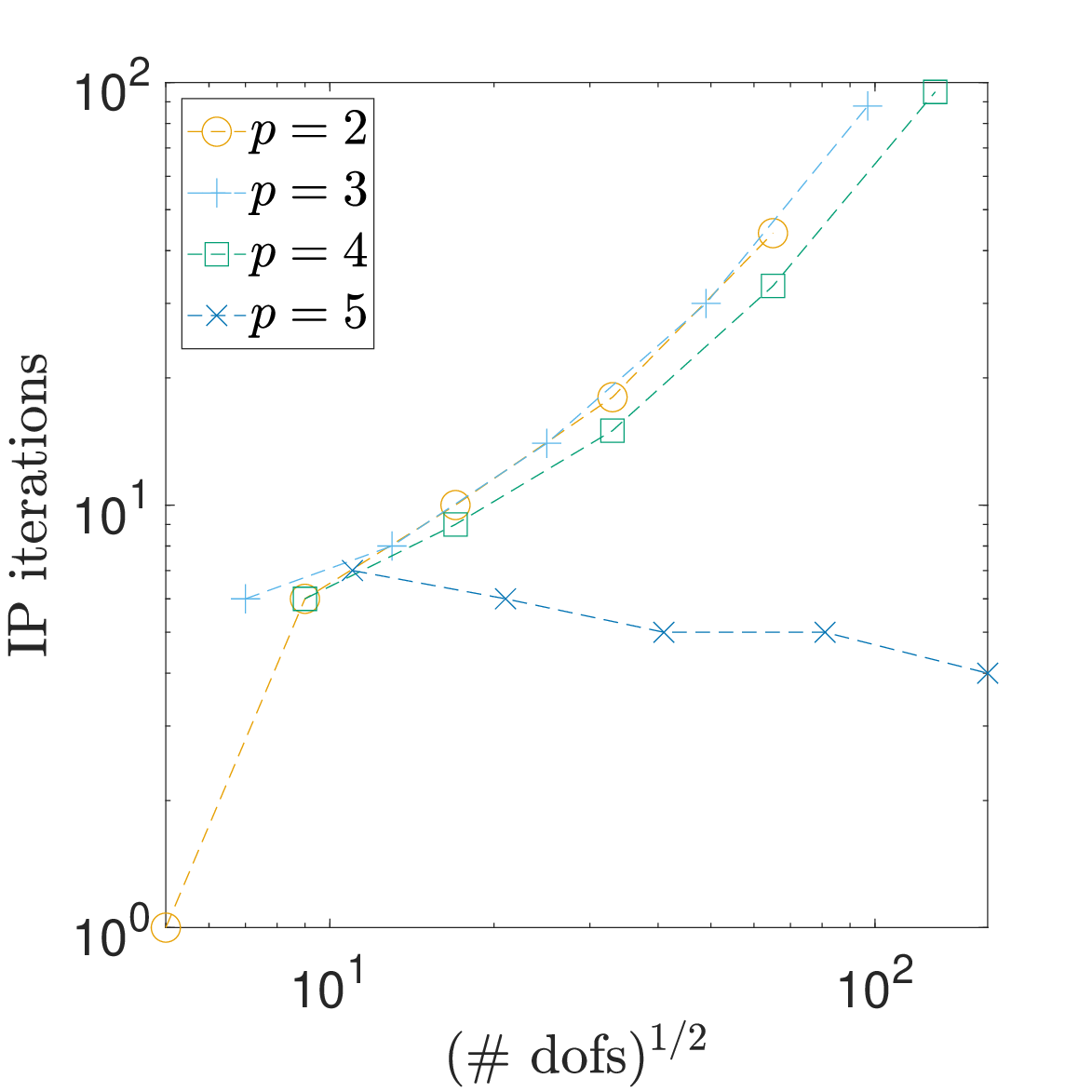}
		\caption{}
		\label{fig:low-order-ipiters}
	\end{subfigure}
	\caption{(a) $H^2(\Omega)$-errors and (b) iterated penalty (IP) iterations for the biharmonic projection problem \cref{eq:biharmonic-projection} with degree $p$ Morgan-Scott elements, $2 \leq p \leq 5$.}
\end{figure}

\subsection{Code availability}

All numerical examples have been implemented using Firedrake \cite{FiredrakeUserManual,Nixonhill23}. The code is available at \cite{zenodo/Firedrake-20231103.0}.

\section{Conclusion}
\label{sec:conclusion}

\Cref{alg:abstract-method} enables one to compute the $H^2$-conforming approximation of \cref{eq:biharmonic-primal-fem} without having to implement $C^1$-conforming finite elements provided that three basic conditions \ref{hp:wp-constants-boundary-assumption}, \ref{hp:tildewp-condition}, and \ref{hp:stokes-complex-bcs-discon-fem} are satisfied. It was shown that all common triangular $C^1$-finite elements satisfy these conditions. \Cref{alg:abstract-method} consists of (i) a pre-processing $H^1$-conforming elliptic projection, (ii) a Stokes-like solve, and (iii) a post-processing $H^1$-conforming elliptic projection, all of which only require finite element spaces and routines that are routinely available in software packages. 

Our approach has similarities to \cite{Gallistl17} in the special case where $\Gamma_{cs}$ is connected although, as shown here, \cite{Gallistl17} will only recover the $C^1$-conforming finite element approximation of the original problem if conditions \ref{hp:wp-constants-boundary-assumption}, \ref{hp:tildewp-condition}, and \ref{hp:stokes-complex-bcs-discon-fem} happen to be satisfied. \Cref{sec:gen-boundary} deals with the case when $\Gamma_{cs}$ is not connected ($N > 1$), which brings a number of fresh difficulties that seem to have not been addressed previously. Nevertheless, the basis idea of \cref{sec:simple-gammacs-connected} is extended to the non-connected case in \cref{alg:abstract-method}.

While there are many other works that seek to relax the $C^1$-continuity required by a conforming method by rewriting the biharmonic and/or Kirchhoff plate problem as a mixed formulation (see e.g. \cite{Amara02,Bramble83,Chen18,Ciarlet74,Veiga07,Farrell22,Rafetseder18} and references therein), it appears that \cref{alg:abstract-method} is the first that is able to deliver the actual $H^2$-conforming approximation defined in \cref{eq:biharmonic-primal-fem}.  

\section*{Acknowledgments}

We would like to thank Aaron Baier-Reinio and Pablo Brubeck for reviewing the code and their helpful suggestions. We additionally thank Pablo Brubeck for providing an early version of Firedrake code to compute the inverse of a symmetric rank $k$ correction of a symmetric positive definite matrix.

\appendix

\section{Properties of the Stokes complex involving the Morgan-Scott elements}

The spaces associated with the Morgan-Scott elements $W_{\Gamma}^p$, $\tilde{W}_{\Gamma}^p$, and $\bdd{G}^{p-1}_{\Gamma}$ were defined in \cref{eq:wp-morgan-scott,eq:tildewp-morgan-scott,eq:thetap-morgan-scott}, respectively. The first task is to characterize the space $\rot \bdd{G}^{p-1}_{\Gamma}$, which turns out to be more complicated than one may expect. In fact, $\rot \bdd{G}_{\Gamma}^{p-1}$ is related to a finite element subspace $R_{\Gamma}^{p-2} \subset L^2(\Omega)$, $p \geq 2$, that satisfies non-standard conditions at some of the element vertices (which we denote by $\mathcal{V}^{\sharp}$, see below), as follows:
\begin{align}
	\label{eq:rp-definition}
	R_{\Gamma}^{p-2} := \left\{ r \in L^2_{\Gamma}(\Omega) : r|_{K} \in \mathcal{P}_{p-2}(K) \ \forall K \in \mathcal{T}, \ \sum_{i=1}^{|\mathcal{T}_{\bdd{a}}|} (-1)^{i} r|_{K_i}(\bdd{a}) = 0 \ \forall \bdd{a} \in \mathcal{V}^{\sharp} \right\},
\end{align}
where 
\begin{align}
	\label{eq:l2gamma-def}
	L^2_{\Gamma}(\Omega) := \{ r \in L^2(\Omega) : (r, 1) = 0 \text{ if } |\Gamma_f| = 0 \}.
\end{align}
The labeling used in the definition \cref{eq:rp-definition} for the elements abutting a vertex $\bdd{a}$ is shown in \cref{fig:patch-schema}, where $\mathcal{V}^{\sharp}$ is the set of mesh vertices located in the interior of $\Omega$ or the interior of $\Gamma_c$ for which all element edges abutting the vertex lie on exactly two lines. 
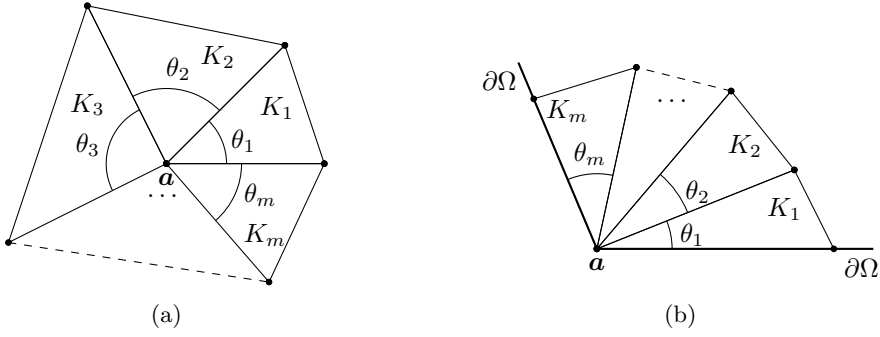
\begin{figure}[htb]
	\centering
	\begin{subfigure}[b]{0.45\linewidth}
		\centering
		\begin{tikzpicture}[scale=0.525]
			
			\coordinate (a) at (0, 0);
			\coordinate (a0) at (4, 0);
			\coordinate (a1) at (3, 3);
			\coordinate (a2) at (-2, 4);
			\coordinate (a3) at (-4, -2);
			\coordinate (a4) at (2.6, -3);				
			
			\filldraw (a) circle (2pt) node[align=center,below]{$\bdd{a}$}
			-- (a0) circle (2pt) 	
			-- (a1) circle (2pt) 
			-- (a);
			\filldraw (a) circle (2pt) node[align=center,below]{}	
			-- (a1) circle (2pt) 
			-- (a2) circle (2pt) 
			-- (a);
			\filldraw (a) circle (2pt) node[align=center,below]{}	
			-- (a2) circle (2pt) 
			-- (a3) circle (2pt) 
			-- (a);
			\filldraw (a) circle (2pt) node[align=center,below]{}	
			-- (a4) circle (2pt) 
			-- (a0) circle (2pt) 
			-- (a);
			\draw[dashed] (a3) -- (a4);
			\draw (2.8, 1.4) node(K0){$K_1$};
			\draw (1.3, 2.7) node(K1){$K_2$};
			\draw (-2, 1.5) node(K1){$K_3$};
			\draw (-0.1, -0.8) node(Kdots){$\ldots$};
			\draw (2.5, -1.8) node(Km){$K_m$};
			\pic["$\theta_1$"{anchor=west}, draw, angle radius=0.8cm, angle eccentricity=1] {angle=a0--a--a1};
			\pic["$\theta_2$"{anchor=south}, draw, angle radius=1cm, angle eccentricity=1] {angle=a1--a--a2};
			\pic["$\theta_3$"{anchor=east}, draw, angle radius=0.8cm, angle eccentricity=1] {angle=a2--a--a3};
			\pic["$\theta_m$"{anchor=west}, draw, angle radius=1cm, angle eccentricity=1] {angle=a4--a--a0};
			
		\end{tikzpicture}
		\caption{}
		\label{fig:internal schema}
	\end{subfigure}
	\hfill
	\begin{subfigure}[b]{0.5\linewidth}
		\centering
		\begin{tikzpicture}[scale=0.525]
			\filldraw (0,0) circle (2pt) node[align=center,below]{$\bdd{a}$}
			-- (6,0) circle (2pt) node[align=center,below]{}	
			-- (5,2) circle (2pt) node[align=center,above]{}
			-- (0,0);
			\filldraw (0,0) circle (2pt) node[align=center,below]{}	
			-- (5,2) circle (2pt) node[align=center,above]{}
			-- (3.4,4) circle (2pt) node[align=center,below]{}
			-- (0,0);
			\filldraw (0,0) circle (2pt) node[align=center,below]{}	
			-- (3.4,4) circle (2pt) node[align=center,above]{};
			\filldraw (1,4.6) circle (2pt) node[align=center,below]{}
			-- (0,0);
			\filldraw (0,0) circle (2pt) node[align=center,below]{}	
			-- (1,4.6) circle (2pt) node[align=center,above]{}
			-- (-1.6,3.8) circle (2pt) node[align=center,below]{}
			-- (0,0);
			
			\draw[dashed] (3.4,4) -- (1,4.6);
			
			\coordinate(a) at (0, 0);
			\coordinate(a0) at (6, 0);
			\coordinate(a1) at (5, 2);
			\coordinate(a2) at (3.4, 4);
			\coordinate(a3) at (1, 4.6);
			\coordinate(a4) at (-1.6, 3.8);
			
			\coordinate (a12) at ($(a)!2/3!(a1)$);
			\coordinate (a121) at ($(a)!2/3-1/sqrt(29)!(a1)$);
			\coordinate (a1205) at ($(a)!2/3-0.5/sqrt(29)!(a1)$);
			
			\coordinate (a22) at ($(a)!2/3!(a2)$);
			\coordinate (a221) at ($(a)!2/3-1/sqrt(27.56)!(a2)$);
			\coordinate (a2205) at ($(a)!2/3-0.5/sqrt(27.56)!(a2)$);
			
			\coordinate (a32) at ($(a)!2/3!(a3)$);
			\coordinate (a321) at ($(a)!2/3-1/sqrt(22.16)!(a3)$);
			\coordinate (a3205) at ($(a)!2/3-0.5/sqrt(22.16)!(a3)$);
			
			\coordinate (a42) at ($(a)!2/3!(a4)$);
			\coordinate (a421) at ($(a)!2/3-1/sqrt(17)!(a4)$);
			\coordinate (a4205) at ($(a)!2/3-0.5/sqrt(17)!(a4)$);

			\pic["$\theta_1$"{anchor=west}, draw, angle radius=1cm, angle eccentricity=1] {angle=a0--a--a1};
			\pic["$\theta_2$"{anchor=west}, draw, angle radius=1.3cm, angle eccentricity=1] {angle=a1--a--a2};
			\pic["$\theta_m$"{anchor=south}, draw, angle radius=1cm, angle eccentricity=1] {angle=a3--a--a4};
			
			\draw (4.75, 1) node(K0){$K_1$};
			\draw (3.75, 2.6) node(K1){$K_2$};
			\draw (1.9, 3.75) node(Kdots){$\ldots$};
			\draw (-0.75, 3.5) node(Km){$K_m$};
			
			\filldraw[thick] (0, 0) -- (7,0);
			\filldraw[thick] (0, 0) -- ($(-1.6, 3.8) + ($(a)!1/sqrt(17)!(a4)$) $);
			
			\draw (6.7, 0) node[align=center,below](G0){$\partial \Omega$};
			\draw ($(-1.6, 3.8) + ($(a)!0.5/sqrt(17)!(a4)$) $) node[align=center,left](Gm1){$\partial \Omega$};
		\end{tikzpicture}		
		\caption{}
		\label{fig:boundary schema}
	\end{subfigure}
	\caption{Notation for mesh around (a) an internal vertex $\bdd{a}$ and (b) a boundary vertex $\bdd{a}$, each abutting $m = |\mathcal{T}_{\bdd{a}}|$ elements.}
	\label{fig:patch-schema}
\end{figure}

A key quantity introduced in \cite{AinCP21LE,ScottVog85,Vogelius83divinv} related to the space $R_{\Gamma}^{p-2}$ is the function $\xi$ defined on the sets $\mathcal{V}_I$, respectively $\mathcal{V}_c$, consisting of mesh vertices located in the interior of $\Omega$, respectively on the interior of $\Gamma_c$, by the rule:
\begin{align*}
	\xi(\bdd{a}) = \sum_{i=1}^{|\mathcal{T}_{\bdd{a}}| - \eta_{\bdd{a}}} |\sin(\theta_i + \theta_{i+1})| \qquad \forall \bdd{a} \in \mathcal{V}_I \cup \mathcal{V}_c,
\end{align*}
where $\mathcal{T}_{\bdd{a}}$ denotes the set of elements abutting $\bdd{a}$ labeled as in \cref{fig:patch-schema}. The quantity $\eta_{\bdd{a}}$ equals 1 if $\bdd{a}$ lies on the domain boundary and equals 0 otherwise. Observe that $\xi(\bdd{a}) = 0$ if $\bdd{a} \in \mathcal{V}^{\sharp} \subseteq \mathcal{V}_I \cup \mathcal{V}_c$, where the set $\mathcal{V}^{\sharp}$ appears in the definition of $R_{\Gamma}^{p-2}$. We then define a corresponding quantity $\xi_{\mathcal{T}}$ for a mesh $\mathcal{T}$ to be
\begin{align}
	\label{eq:xi-mesh-def}
	\xi_{\mathcal{T}} = \min_{\bdd{a} \in (\mathcal{V}_I \cup \mathcal{V}_c) \setminus \mathcal{V}^{\sharp}} \xi(\bdd{a}) > 0.
\end{align}
Of course, it is possible for some families of meshes that this quantity can become arbitrarily small. However, it is a relatively simple matter to adjust such meshes to ensure that $\xi_{\mathcal{T}}$ remains bounded away from zero uniformly. For instance, in the case of an interior vertex, one can perturb the location of the offending vertex while preserving the mesh topology; see e.g. \cite[p. 35 Remark 2]{AinCP21LE} for other remedies. Despite the apparently rather obscure definition of $R_{\Gamma}^{p-2}$, the following result  \cite{AinCP21LE,ScottVog85,Vogelius83divinv} shows that $R_{\Gamma}^{p-2}$ and $\rot \bdd{G}^{p-1}_{\Gamma}$ coincide when $p \geq 5$:
\begin{lemma}
	\label{lem:invert-rot-rp}
	Let $p \geq 5$. Then $\rot \bdd{G}^{p-1}_{\Gamma} = R_{\Gamma}^{p-2}$ and for each $r \in R_{\Gamma}^{p-2}$, there exists $\bdd{\theta} \in \bdd{G}^{p-1}_{\Gamma}$ satisfying
	\begin{align}
		\label{eq:invert-rot-vp}
		\rot \bdd{\theta} = r \quad \text{and} \quad \|\bdd{\theta}\|_1 \leq C \xi_{\mathcal{T}}^{-1} \|r\|,
	\end{align}
	where $C > 0$ is independent of $r$, $h$, $p$, and $\xi_{\mathcal{T}}$.
\end{lemma}
\begin{proof}
	Thanks to \cite[Theorem 4.1]{AinCP21LE}, $R_{\Gamma}^{p-2} = \dive \tilde{\bdd{G}}_{\Gamma}^{p-1}$, where
	\begin{align}
		\label{eq:tilde-theta-def}
		\tilde{\bdd{G}}_{\Gamma}^{p-1} &:= \{ \bdd{v} \in \bdd{C}^0(\Omega) : \bdd{v}|_{\Gamma_c} = \bdd{0}, \ \bdd{v}\cdot \unitvec{n}|_{\Gamma_s} = 0, \text{ and }  \bdd{v}|_K \in \mathcal{P}_{p-1}(K)^2 \ \forall K \in \mathcal{T} \}. 
	\end{align}
	The identity $R_{\Gamma}^{p-2} = \dive \tilde{\bdd{G}}_{\Gamma}^{p-1} = \rot \bdd{G}^{p-1}_{\Gamma}$ is an immediate consequence of the fact that $\bdd{u} \in \bdd{G}^{p-1}_{\Gamma}$ if and only if $\bdd{v} = (-u_2, u_1) \in \tilde{\bdd{G}}_{\Gamma}^{p-1}$ and $\rot \bdd{u} = \dive \bdd{v}$.
	
	For every $r \in R_{\Gamma}^{p-2}$, there exists $\bdd{v} \in \tilde{\bdd{G}}_{\Gamma}^p$ satisfying $\dive \bdd{v} = r$ and $\|\bdd{v}\|_1 \leq C \xi_{\mathcal{T}}^{-1} \|r\|$ by \cite[Theorem 5.1]{AinCP21LE}. The function $\bdd{\theta} = (-v_2, v_1) \in \bdd{G}^{p-1}_{\Gamma}$ then satisfies \cref{eq:invert-rot-vp}, which completes the proof.
\end{proof}

\begin{remark}
	\Cref{lem:invert-rot-rp} is sharp in the sense that for $p \leq 4$, there exist meshes such that $\rot \bdd{G}^{p-1}_{\Gamma} \subsetneq R_{\Gamma}^{p-2}$; see e.g. section 7 of \cite{ScottVog84}.
\end{remark}

\subsection{Uniform bounds on the inf-sup constant in \cref{eq:stokes-morgan-scott-gen-inf-sup}}

We begin with the following result at the continuous level:
\begin{lemma}
	\label{lem:invert-rot-free-averages}
	For every $r \in L^2_{\Gamma}(\Omega)$ and $\vec{\omega} \in \mathbb{R}^N$ satisfying \cref{eq:kappaij-constraint}, there exists $\bdd{\theta} \in \bdd{\Theta}_{\Gamma}(\Omega)$ satisfying
	\begin{align}
		\label{eq:invert-rot-free-averages}
		\rot \bdd{\theta} = r, \quad (\unitvec{t} \cdot \bdd{\theta}, 1)_{\Gamma_f^{(i)}} = \omega_{i}, \quad 1 \leq i \leq N, \quad \text{and} \quad \|\bdd{\theta}\|_{1} \leq C\left( \|r\| + |\vec{\omega}| \right),
	\end{align}
	where $C > 0$ is independent of $r$ and $\vec{\omega}$.
\end{lemma}
\begin{proof}
	Let $r \in L^2_{\Gamma}(\Omega)$ be given. By \cite[Lemma A.1]{AinCP21LE}, there exists $\bdd{v} \in \bdd{\Theta}_{\Gamma}(\Omega)$ satisfying $\rot \bdd{v} = q$ and $\|\bdd{v}\|_{1} \leq C \|r\|$. We now construct a rot-free correction to $\bdd{v}$ to interpolate the values $\omega_{i}$ on $\Gamma_f^{(i)}$.

	For $1 \leq i \leq N$, we define functions $\bdd{g}_{i} \in \bdd{L}^2(\Gamma_f^{(i)})$ as follows. Let $\Gamma_k \subset \Gamma_f^{(i)}$ be an edge of $\Gamma$ and set
	\begin{align*}
		\bdd{g}_{i}:= \begin{cases}
			[\omega_{i} - (\unitvec{t} \cdot \bdd{v}, 1)_{\Gamma_f^{(i)}} ] \phi_k \unitvec{t} & \text{on } \Gamma_k, \\
			\bdd{0} & \text{on } \Gamma_f^{(i)} \setminus \Gamma_k,
		\end{cases}
	\end{align*}
	where $\phi_k \in C^{\infty}_c(\Gamma_k)$ is any smooth function satisfying $(\phi_k, 1)_{\Gamma_{k}} = 1$. The function $\bdd{g} \in \bdd{L}^2(\partial \Omega)$ defined by
	\begin{align*}
		\bdd{g} = \begin{cases}
			\bdd{g}_{i} & \text{on } \Gamma_f^{(i)}, \ 1 \leq i \leq N, \\
			\bdd{0} & \text{otherwise},
		\end{cases}
	\end{align*}
	then satisfies $\bdd{g} \in \bdd{H}^{1/2}(\Gamma)$ and
	\begin{align*}
		\int_{\Gamma} \unitvec{t} \cdot \bdd{g} \ ds = \sum_{i=1}^{N} \omega_{i} - \int_{\Gamma} \unitvec{t} \cdot \bdd{v} \ ds = \int_{\Omega} (r - \rot \bdd{v}) \ d\bdd{x} = 0,
	\end{align*}
	where we used \cref{eq:kappaij-constraint}. \cite[Lemma 2.2, p. 24]{GiraultRaviart86} asserts the existence of $\bdd{u} \in \bdd{H}^1(\Omega)$ satisfying 
	\begin{align*}
		\rot \bdd{u} \equiv 0, \quad \bdd{u}|_{\Gamma} = \bdd{g}, \quad \text{and} \quad \|\bdd{u}\|_{1} \leq C \|\bdd{g}\|_{1/2, \Gamma},
	\end{align*}
	where $\|\cdot\|_{1/2, \Gamma}$ is the $H^{1/2}(\Gamma)$ norm. In particular, $\bdd{u} \in \bdd{\Theta}_{\Gamma}(\Omega)$. Consequently, the function $\bdd{\theta} := \bdd{v} + \bdd{u}$ then satisfies $\bdd{\theta} \in \bdd{\Theta}_{\Gamma}(\Omega)$,
	\begin{align*}
		\rot \bdd{\theta} = r, \quad (\unitvec{t} \cdot \bdd{\theta}, 1)_{\Gamma_f^{(i)}} = \omega_{i}, \quad 1 \leq i \leq N,
	\end{align*}
	and $\| \bdd{\theta}\|_{1} \leq \| \bdd{v} \|_1 + \|\bdd{u}\|_{1} 	\leq C (\|r\| + \|\bdd{g}\|_{1/2, \Gamma}) \leq C (\|r\| + |\vec{\omega}|)$.
\end{proof}

\noindent The discrete analogue of \cref{lem:invert-rot-free-averages} is the following:
\begin{lemma}
	\label{lem:invert-rot-free-averages-discrete}
	Let $p \geq 5$. For every $r \in R_{\Gamma}^{p-2} = \rot \bdd{G}^{p-1}_{\Gamma}$ and $\vec{\omega} \in \mathbb{R}^N$ satisfying \cref{eq:kappaij-constraint}, there exists $\bdd{\theta} \in \bdd{G}^{p-1}_{\Gamma}$ satisfying 
	\begin{align}
		\label{eq:invert-rot-free-averages-discrete}
		\rot \bdd{\theta} = r, \quad (\unitvec{t} \cdot \bdd{\theta}, 1)_{\Gamma_f^{(i)}} = \omega_{i}, \quad 1 \leq i \leq N, \quad \text{and} \quad  	\|\bdd{\theta}\|_{1} \leq C \left( \xi_{\mathcal{T}}^{-1} \|r\| + |\vec{\omega}|\right),
	\end{align}
	where $C > 0$ is independent of $r$, $\vec{\omega}$, $h$, $p$, and $\xi_{\mathcal{T}}$.
\end{lemma}
\begin{proof}
	Let $r \in R_{\Gamma}^{p-2}$ and $\vec{\omega} \in \mathbb{R}^N$ be as in the statement of the lemma.
	
	\textbf{Step 1.} Thanks to \cite[Corollary 5.1]{AinCP21LE}, there exists $\tilde{\bdd{\phi}} \in \tilde{\bdd{\Theta}}_{\Gamma}^{p-1}$ satisfying the following on all $K \in \mathcal{T}$:
	\begin{align*}
		\dive \tilde{\bdd{\phi}}|_{K}(\bdd{a}) = r|_{K}(\bdd{a}) \quad \forall \bdd{a} \in \mathcal{V}_K \quad \text{and} \quad h_K^{-1} \| \tilde{\bdd{\phi}} \|_K + | \tilde{\bdd{\phi}} |_{1, K} \leq C \xi_{\mathcal{T}}^{-1} \|r\|_{K},
	\end{align*}
	where $\tilde{\bdd{\Theta}}_{\Gamma}^{p-1}$ is defined in \cref{eq:tilde-theta-def} and $\mathcal{V}_K$ denotes the vertices of $K$. The function $\bdd{\phi} = (-\tilde{\phi}_2, \tilde{\phi}_1)$ satisfies
	\begin{align*}
		\rot \bdd{\phi}|_{K}(\bdd{a}) = r|_{K}(\bdd{a}) \quad \forall \bdd{a} \in \mathcal{V}_K \quad \text{and} \quad h_K^{-1} \| \bdd{\phi} \|_K + | \bdd{\phi} |_{1, K} \leq C \xi_{\mathcal{T}}^{-1} \|r\|_{K}.
	\end{align*}
	
	\textbf{Step 2.}  Let $\bdd{\Psi} \in \bdd{\Theta}_{\Gamma}(\Omega)$ be given by \cref{lem:invert-rot-free-averages} and let $\Pi_{SZ} : \bdd{\Theta}_{\Gamma}(\Omega) \to \bdd{\Theta}_{\Gamma}^{1}$ denote the linear Scott-Zhang projection \cite{Scott90}. We define $\bdd{u} \in \bdd{\Theta}_{\Gamma}^4$ by assigning degrees of freedom as follows:
	\begin{alignat*}{2}
		\bdd{u}(\bdd{a}) &= \Pi_{SZ} \bdd{\Psi}(\bdd{a}) \qquad & &\forall \bdd{a} \in \mathcal{V}, \\
		\grad \bdd{u}(\bdd{a}) &= \bdd{0} \qquad & &\forall \bdd{a} \in \mathcal{V}, \\
		\int_{\gamma} \bdd{u} \ ds &= \int_{\gamma} (\bdd{\Psi} - \bdd{\phi}) \ ds \qquad & &\forall \gamma \in \mathcal{E}, \\
		\int_{K} \bdd{u} \cdot \bdd{q} \ d\bdd{x} &= \int_{K} (\bdd{\Psi} - \bdd{\phi}) \cdot \bdd{q} \ ds \qquad & &\forall \bdd{q} \in [\mathcal{P}_0(K)]^2, \ \forall K \in \mathcal{T}.
	\end{alignat*}
	Arguing as in the proof of \cite[Lemma 5.1]{AinCP21LE}, the function $\bdd{\eta} = \bdd{\phi} + \bdd{u}$ satisfies $\bdd{\eta} \in \bdd{\Theta}_{\Gamma}^{p}$,
	\begin{align*}
		\rot \bdd{\eta}|_{K}(\bdd{a}) = r|_{K}(\bdd{a}) \quad \forall \bdd{a} \in \mathcal{V}_K \quad \text{and} \quad \int_{K} \rot \bdd{\eta} \ d\bdd{x} = \int_{K} r \ d\bdd{x} \quad \forall K \in \mathcal{T}.
	\end{align*}
	Additionally,
	\begin{align*}
		\int_{\gamma} \bdd{\eta} \ ds = \int_{\gamma} \bdd{\Psi} \ ds \quad \forall \gamma \in \mathcal{E} \quad \implies \quad (\unitvec{t} \cdot \bdd{\eta}, 1 )_{\Gamma_f^{(i)}} = \omega_{i}, \quad 1 \leq i \leq N,
	\end{align*}
	where $\mathcal{E}$ denotes the edges of the mesh. Moreover, arguing as in the proof of \cite[Lemma 3.7]{AinCP19StokesI}, we may show that 
	\begin{align*}
		\| \bdd{\eta} \|_{1} \leq C \left( \| \bdd{\phi}\|_{1} + \|\bdd{\Psi}\|_{1} \right) \leq C \left( \xi_{\mathcal{T}}^{-1} \|r\| + |\vec{\omega}| \right).
	\end{align*}

	\textbf{Step 3.} For each $K \in \mathcal{T}$, there exists $\tilde{\bdd{\psi}}^K \in [\mathcal{P}_{p-1}(K)]^2 \cap \bdd{H}^1_0(K)$ satisfying
	\begin{align*}
		\dive \tilde{\bdd{\psi}}^K = (r - \rot \bdd{\eta})_{K} \quad \text{and} \quad \| \tilde{\bdd{\psi}}^K \|_{1, K} \leq C \left( \|r\|_K + \|\bdd{\eta}\|_{1, K} \right)
	\end{align*}
	thanks to \cite[Theorem 3.4]{AinCP19StokesI}. Thus, the function $\bdd{\theta} \in \bdd{G}^{p-1}_{\Gamma}$ defined elementwise by
	\begin{align*}
		\bdd{\theta}|_{K} := \bdd{\eta}|_{K} + (-\tilde{\psi}^K_2, \tilde{\psi}^K_{1}) \qquad \text{on } K
	\end{align*}
	satisfies \cref{eq:invert-rot-free-averages-discrete}.
\end{proof}	

An immediate consequence of \cref{lem:invert-rot-free-averages-discrete} is a uniform in $h$ and $p$ lower bound on the inf-sup constant $\beta_p$ \cref{eq:stokes-morgan-scott-gen-inf-sup}:
\begin{corollary}
	\label{cor:stokes-morgan-scott-gen-inf-sup}
	For every $r \in \rot \bdd{G}^{p-1}_{\Gamma}$ and $\vec{\kappa} \in \mathbb{R}^{N-1}$, there exists $\bdd{\theta} \in \bdd{G}^{p-1}_{\Gamma}$ satisfying
	\begin{align}
		\rot \bdd{\theta} = r, \quad \int_{\Gamma_f^{(i)}} \unitvec{t} \cdot \bdd{\theta} \ ds = \kappa_i, \quad 1 \leq i \leq N-1, \quad \text{and} \quad \|\bdd{\theta}\|_{1} \leq C \left( \xi_{\mathcal{T}}^{-1} \|r\| + |\vec{\kappa}| \right),
	\end{align}
	where $C > 0$ is independent of $r$, $\vec{\kappa}$, $h$, $p$, and $\xi_{\mathcal{T}}$. Consequently, the spaces $\bdd{G}^{p-1}_{\Gamma} \times (\rot \bdd{G}^{p-1}_{\Gamma} \times \mathbb{R}^{N-1})$ are uniformly inf-sup stable:
	\begin{align}
		\label{eq:stokes-morgan-scott-inf-sup-bound}
		\beta_X \geq \xi_{\mathcal{T}} \beta_0,
	\end{align}
	where $\beta_0 > 0$ is independent of $h$, $p$, and $\xi_{\mathcal{T}}$ and $\beta_X$ is defined in \cref{eq:stokes-morgan-scott-gen-inf-sup}.
\end{corollary}
\begin{proof}
	Let $r \in \rot \bdd{G}^{p-1}_{\Gamma}$ and $\vec{\kappa} \in \mathbb{R}^{N-1}$ be given. Using the notation in the statement of \cref{lem:invert-rot-free-averages}, we set
	\begin{align*}
		\omega_i := \kappa_i, \qquad 1 \leq i \leq N-1, \quad \text{and} \quad \omega_{N} := \int_{\Omega} r \ d\bdd{x} -\sum_{j=1}^{N-1} \kappa_{j}.
	\end{align*}
	The result now follows from \cref{lem:invert-rot-free-averages-discrete} and that $\xi_{\mathcal{T}} \leq C$ by shape regularity.
\end{proof}

\bibliographystyle{siamplain}

\end{document}